\documentclass[11pt, a4paper]{article}

\usepackage{geometry}                		
\usepackage{amsmath, amsfonts, amssymb, amsthm, bm}
\usepackage{graphicx, subfigure, tikz, booktabs}
\usepackage{epstopdf}
\usepackage{cleveref}
\ifpdf
  \DeclareGraphicsExtensions{.eps,.pdf,.png,.jpg}
\else
  \DeclareGraphicsExtensions{.eps}
\fi

\newtheorem{theorem}{Theorem}
\newtheorem{problem}{Problem}
\newtheorem{definition}{Definition}

\begin{document}

\title{Recovering the sources in the stochastic wave equations from multi-frequency far-field patterns}


\author{Yan Chang\thanks{School of Mathematics, Harbin Institute of Technology, Harbin, P. R. China. ({21B312002@stu.hit.edu.cn}).},
            Yukun Guo\thanks{School of Mathematics, Harbin Institute of Technology, Harbin, P. R. China. ({ykguo@hit.edu.cn}).},
	   Zhipeng Yang\thanks{School of Mathematics and Statistics, Lanzhou University, Lanzhou, P. R. China. ({yangzp@lzu.edu.cn}).},
	  and Yue Zhao\thanks{School of Mathematics and Statistics, and Key Lab NAA–MOE, Central China Normal University, Wuhan, P. R. China. ({zhaoy@ccnu.edu.cn}).}}
	  
\date{}

\maketitle

\begin{abstract}
   This paper concerns the inverse source scattering problems of recovering random sources for acoustic and elastic waves. The underlying sources are assumed to be random functions driven by an additive white noise. The inversion process aims to find the essential statistical characteristics of the mean and variance from the radiated random wave field at multiple frequencies. To this end, we propose a non-iterative algorithm by approximating the mean and variance via the truncated Fourier series. Then, the Fourier coefficients can be explicitly evaluated by sparse far-field measurements, resulting in an easy-to-implement and efficient approach for the reconstruction. Demonstrations with extensive numerical results are presented to corroborate the feasibility and robustness of the proposed method.
\end{abstract}



\section{Introduction}

It is widely acknowledged that random phenomena occur naturally in wavefield propagation \cite{Evans13}. The mechanical fields, such as acoustic or elastic waves, physically originate from certain energy-input excitation. In principle, the output of the excitation source can not be absolutely stable or deterministic because the oscillation of the source is inevitably affected by unpredictable perturbations. In such scenarios, random sources emanate random wavefields with inescapable uncertainties. Hence, wave-detecting devices would always record this indeterministic data. Conversely, inferring the inaccessible information of the unknown source from the radiating statistical data formulates the inverse stochastical source problems (ISSP). 

The ISSPs have significant applications in diverse areas, encompassing geophysical exploration, antenna design and synthesis, medical imaging, and nondestructive testing \cite{WAC03, A99, B01, BGRS19, CCM21, HFG21, MLO21}. Nevertheless, due to the inherent randomness and ill-posedness, precisely formulating the ISSPs and developing effective algorithms are both challenging \cite{LLW22, KS05}. For instance, in comparison with the deterministic settings, in stochastic inverse problems, the source may be extremely rough to hinder the pointwise existence of the solution. Therefore, the solution to stochastic problems should be usually understood in the sense of distribution. Moreover, owing to its stochastic nature, a single solution to the inverse problems for a particular realization of the randomness is typically purposeless. Instead, reconstructing the target source's pivotal statistical characteristics, such as the mean and variance, is practically more valuable in quantifying the uncertainties. Ensuring high-quality evaluation of these statistical indices relies severely on the amount of data, computational resources, and efficient algorithms.

Remarkable progress has been made in stochastic inverse scattering problems, ranging from theoretical analysis to numerical approaches. It is theoretically demonstrated in \cite{LLW22} by ergodicity that the principal symbols of the covariance and relation operators can be uniquely determined by a single realization of the far-field pattern averaged over the frequency band with probability one. Several attempts to resolve the one-dimensional stochastic inverse source problems can be found in \cite{10, 6, 27}, where the governing equations are stochastic ordinary differential equations. Note that analogously extending these methods to the multi-dimensional case is non-trivial. Later, concerning the inverse source problem for the 2D and 3D stochastic Helmholtz equation, Bao et al. \cite{BCL16, BCL17} developed a regularized Kaczmarz method by resorting to multifrequency scattering data to reconstruct the mean and the variance of the stochastic source. Recently, \cite{LL24} shows that the ISSPs at a fixed frequency generally lack stability, namely, a mild variation in the data may result in drastic inaccuracies in the reconstruction. We also refer to \cite{LLM19, LLM21} for studies on the inverse source problems of the random Schr\"{o}dinger equations.
 
In the current study, we investigate the inverse source problem for the stochastic acoustic and elastic wave equations. The source is assumed to be a random function driven by a white noise whose correlation function is the delta function. The inverse problem is then reformulated as the recovery of the mean and variance of the random source. This work aims to develop a unified and effective numerical scheme to reconstruct the random sources of the acoustic and elastic waves from the far-field pattern, which, to our knowledge, has not been tackled in the literature.

Based on the Fourier expansions, we propose a novel numerical scheme to recover the source from the multi-frequency data. Our study is motivated by the Fourier method for reconstructing the various deterministic source functions in acoustics \cite{IP15, IP17}, electromagnetics \cite{JDE18}, and elasticity mechanics \cite{IPI22, JCAM19, CGYZ24, CGZ24}. The key idea of the Fourier method is to approximate the source function by a Fourier expansion and then establish an explicit computational formula of the Fourier coefficients. Then, the source function can be approximated by truncated Fourier expansion once the Fourier coefficients are determined. A notable advantage of this method is that the Fourier coefficients can be derived directly from the measured radiated field by choosing appropriate admissible frequencies, and no iterative solver is involved in the reconstruction.

Specifically, we propose a non-iterative Fourier-based method to reconstruct the mean and variance of the random source from the expectation and correlation of the acoustic far-field pattern, respectively. In our scheme, the Fourier coefficients of the statistical quantities can be computed directly from the same statistics of the far-field pattern. Superior to the utilization of near-field data, our method only requires limited observation directions for each frequency, so that the total measurement directions are sparse. For the elastic case, the inverse random source problem is more challenging due to the coexistence and interaction of compressional and shear waves propagating at different speeds. To untangle these components and compute the Fourier coefficients, we employ linear combinations of correlations of far-field patterns of compressional and shear waves at appropriately chosen frequencies. In summary, by utilizing the statistical property of the random source, we establish the relation between the Fourier coefficients and the corresponding stochastical measurements, which enables us to recover the deterministic and stochastic components of the source, respectively. Moreover, the unified framework applies stably and efficiently to both acoustic and elastic models. 

The rest of this paper is organized as follows: In the next section, we are concerned with the acoustic problem and develop the Fourier method to reconstruct the source in the stochastic wave equation. Section 3 is devoted to the inverse stochastic source problem for the elastic wave equation. Numerical experiments are conducted in corresponding sections to illustrate the performance of the proposed method. Several conclusion remarks are given in Section \ref{sec: conclusion}.


\section{Acoustic waves}

This section starts with the acoustic model for the direct and inverse stochastic source problems. Then the Fourier method is developed to recover the stochastic source term. Numerical examples will also be included to verify the performance of the proposed method.

\subsection{Problem formulation}
Consider the following two-dimensional Helmholtz equation
\begin{equation}\label{eq: Helmholtz}
    \Delta u+k^2u=-f\quad \text{in}\ \mathbb{R}^2,
\end{equation}
 where $k>0$ is the wavenumber, $u$ is the radiated scalar field. The source current density $f$ in \eqref{eq: Helmholtz} is assumed to be a random function driven by an additive white noise which has the form 
\begin{align}\label{eq: f}
 f=g+\sigma\dot{W}_x.
\end{align}
 Here, $g$ and $\sigma\ge0$ are two deterministic real functions independent of the wavenumber $k$ and both are compactly supported in a square $\Omega\subset\mathbb{R}^2$, $\dot{W}_x$ is a homogeneous white noise with $W(x)$ being the one-dimensional two-parameter Brownian sheet. Here the white noise $\dot{W}_x$ can be viewed as the derivative of $W(x)$.  In addition, the radiated field $u$ is required to satisfy the Sommerfeld radiated condition
\begin{align}\label{eq: Sommerfeld}
   \lim\limits_{r=|x|\to\infty}\sqrt{r}(\partial_r u-\mathrm{i}ku)=0,
\end{align}
 uniformly in all directions $\widehat{x}=x/|x|\in \mathbb{S}$ with $\mathbb{S}=\left\{x\in\mathbb{R}^2:|x|=1\right\}$.
 
Under the assumption that the random source is driven by the incoherent white noise, the variance can be reconstructed. 
 Recall the fundamental solution to the Helmholtz equation 
\begin{align}\label{eq: Phi}
     G_k(x, y)=\frac{\mathrm{i}}{4}H_0^{(1)}(k|x-y|),
\end{align}
where $H_0^{(1)}$ is the Hankel function of the first kind of order zero. Then the mild solution to the Helmholtz equation \eqref{eq: Helmholtz} under the radiation condition \eqref{eq: Sommerfeld} at wavenumber $k$ is given by
\begin{align}\label{eq: radiatedAcoustic}
    u(x; k)=\int_\Omega G_k(x, y)g(y)\mathrm{d}y+\int_\Omega G_k(x, y)\sigma(y)\mathrm{d}W_y.
\end{align}
 
Noticing the asymptotic behavior of the Hankel function \cite{CK19}, the radiated field \eqref{eq: radiatedAcoustic} admits the following asymptotic expansion 
\begin{align}\label{eq: acoustic_asymptotic}
 u(x; k)=\frac{\mathrm{e}^{\mathrm{i}k|x|}}{\sqrt{|x|}}\left\{u^\infty(\widehat{x}; k)+\mathcal{O}\left(\frac{1}{|x|}\right)\right\},\quad |x|\to \infty,
\end{align}
 which holds uniformly with respect to all directions $\widehat{x}=x/|x|$. In \eqref{eq: acoustic_asymptotic}, $u^\infty$ is called the far-field pattern, whose definition is given by
 \begin{align}\label{eq: acoustic_far}
    u^\infty(\widehat{x}; k)=\int_\Omega G_k^\infty(\widehat{x}, y)g(y)\mathrm{d}y+
 \int_\Omega G_k^\infty(\widehat{x}, y)\sigma(y)\mathrm{d}W_y.
 \end{align}
 Here
 \begin{align*}
   G_k^\infty(\widehat{x},y)=\frac{\mathrm{e}^{\mathrm{i}\pi/4}}{\sqrt{8\pi k}}\mathrm{e}^{-\mathrm{i}k\widehat{x}\cdot y},\quad \widehat{x}\in\mathbb{S},
 \end{align*}
 is the far-field pattern to the fundamental solution \eqref{eq: Phi}.

 Taking the expectation on both sides of \eqref{eq: acoustic_far} and using the property (see \cite{BCL16})
 \begin{align*}
 \mathbf{E}\left[\int_\Omega G_k^\infty(\widehat{x},y)\sigma(y)\mathrm{d}W_y
 \right]=0,
 \end{align*}
 we derive that
\begin{align}\label{eq: Efar}
  \mathbf{E}\left[u^\infty(\widehat{x};k)\right]=\int_\Omega G_k^\infty(\widehat{x},y)g(y)\mathrm{d}y,
\end{align}
which can be used to reconstruct the mean $g.$ Meanwhile, from the above equation, we know that once the boundary expectation data of the radiated wave field is available, the stochastic inverse problem can be reformulated as the deterministic inverse problem. 

We denote the covariance of random variables $u$ and $v$ by $\mathbf{C}[u, v]=\mathbf{E}[(u-\mathbf{E}[u])(v-\mathbf{E}[v])]$.
Take some small $k_0>0$, then for each $\tau>0,$ we can calculate the covariance $\mathbf{C}\left[u^\infty(\widehat{x}; k_0+\tau),u^\infty(\widehat{x};k_0)\right]$ on both sides of \eqref{eq: acoustic_far}. By taking the expectation $\mathbf{E}\left[u^\infty(\widehat{x};k)\right]$ and the covariance $\mathbf{C}\left[u^\infty(\widehat{x};k_0+\tau),u^\infty(\widehat{x};k_0)\right]$ to be the measurements, we now propose the inverse scattering problem (ISP) as follows:
\begin{problem}
Choose two sets of wavenumbers $\{k\}$ and $\{\tau\},$ and take some small wavenumber $k_0>0.$ Then the ISP under consideration is to determine the mean $g$ and the variance $\sigma^2$ of the source function from  the multi-frequency far-field statistical measurements 
$$
  \left\{\mathbf{E}\left[u^\infty(\widehat{x};k)\right],
\mathbf{C}\left[u^\infty(\widehat{x}; k_0+\tau),u^\infty(\widehat{x};k_0)\right]\right\}.
$$
\end{problem}

Before delving into the implementation of inversion, we shall introduce several notations and the relevant Sobolev spaces. Without loss of generality, let $a>0$ and define 
$$
   \Omega=\left(-\frac{a}{2},\frac{a}{2}\right)\times\left(-\frac{a}{2},\frac{a}{2}\right)
$$
such that $\mathrm{supp}\, g\subset\Omega$ and $\mathrm{supp}\,\sigma \subset\Omega$. For each $s>0,$ we define the periodic Sobolev space $H^s(\Omega)$ by 
$$
H^s(\Omega)=\left\{v\in L^2(\Omega):\|v\|_s<\infty\right\},
$$
where the norm $\|\cdot\|_s$ is defined by 
$$
\|v\|_s=\left(
\sum_{\bm{l}\in\mathbb{Z}^2}\left(1+|\bm{l}|^2\right)^s|\widehat{v}_{\bm{l}}|^2
\right)^{1/2},
$$
with $\widehat{v}_{\bm{l}}$ being the Fourier coefficients of $v$:
\begin{align}\label{eq: hatv}
\widehat{v}_{\bm{l}}=\frac{1}{a^2}\int_{\Omega}v(x)\overline{\phi_{\bm{l}}(x)}\mathrm{d}x.
\end{align}
In \eqref{eq: hatv}, the overbar denotes the complex conjugate, and $\phi_{\bm{l}}(x)$ is the Fourier basis function defined by
\begin{align}\label{eq: basis}
  \phi_{\bm{l}}(x) =\mathrm{e}^{\mathrm{i}\frac{2\pi}{a}\bm{l}\cdot x},\quad\bm{l}\in\mathbb{Z}^2.
\end{align}

Now, the mean function $g$ and the variance function $\sigma^2$ can be expanded through the Fourier series
\begin{align}\label{eq: g}
  g(x)=\sum_{\bm{l}\in\mathbb{Z}^2}\widehat{g}_{\boldsymbol l}\phi_{\boldsymbol l}(x), \\\label{eq: sigma}
  \sigma^2(x)=\sum_{\bm{l}\in\mathbb{Z}^2}\widehat{\sigma}_{\boldsymbol l}\phi_{\boldsymbol l}(x).
\end{align}
with $\widehat{g}_{\boldsymbol l}, \widehat{\sigma}_{\boldsymbol l}, $ ${\boldsymbol l}\in\mathbb{Z}^2$ being the Fourier coefficients.

So far, we have represented the mean function $g$ and the variance function $\sigma^2$ by the Fourier expansion. Once the Fourier coefficients $\widehat{g}_{\bm l}, \widehat{\sigma}_{\bm l}, $ ${\bm l}\in\mathbb{Z}^2$ are available, we immediately obtain the reconstruction formulas. The computations of the Fourier coefficients will be given in the following two subsections.
 
\subsection{Reconstruct the mean function $g$}
 
This subsection aims to reconstruct the mean function $g$ from multifrequency far-field mean values $\mathbf{E}[u^\infty(\widehat{x};k)]$. Following \cite{IP17}, we introduce the following definition of admissible wavenumbers.

\begin{definition}[Admissible wavenumber to reconstruct the mean function]\label{def: admissible_mean}
   Take a sufficiently small positive constant $\lambda>0$ and denote 
   $$
       \bm{l}_0=(\lambda, 0),
   $$
   then the admissible wavenumber is defined through 
   \begin{align}\label{eq: admissible_mean}
   k_{\bm l}:=
   \begin{cases}
        \dfrac{2\pi}{a}|\bm{l}|, &\bm{l}\in\mathbb{Z}^2\backslash\{\bm{0}\},\vspace{2mm}\\
        \dfrac{2\pi}{a}\lambda, &\bm{l}=\bm{0}.
   \end{cases}
   \end{align}
   The corresponding observation directions are defined by 
   \begin{align}\label{eq: xhat}
   \widehat{x}_{\bm{l}}:=
   \begin{cases}
        \dfrac{\bm l}{|\bm l|}, &{\bm l}\in\mathbb{Z}^2\backslash\{\bm 0\},\\
        (1, 0), & {\bm l}={\bm 0}.
   \end{cases}
   \end{align}
 \end{definition}

Based on \Cref{def: admissible_mean}, we have the following computational formulas for the Fourier coefficients.
 \begin{theorem}
 For $\bm l\in \mathbb{Z}^2$, let the wavenumber $k_{\bm l}$ and the observation direction $\widehat{x}_{\bm l}$ be defined by \eqref{eq: admissible_mean} and \eqref{eq: xhat}, respectively. Then the Fourier coefficients $\{\widehat{g}_{\bm l}\}$ of $g$ in \eqref{eq: g} can be determined by $\mathbf{E}[u^\infty(\widehat{x}_{\bm l}; k_{\bm l})]$.
\end{theorem}

\begin{proof}
From \eqref{eq: Efar} and \eqref{eq: admissible_mean}, we know that for ${\bm l}\ne{\bm 0},$
\begin{align*}
  \mathbf{E}\left[u^\infty(\widehat{x}_{\bm l};k_{\bm l})\right] & =\gamma_{\bm l}\int_\Omega g(y)\mathrm{e}^{-\mathrm{i}k_{\bm l}\widehat{x}_{\bm l}\cdot y}\mathrm{d}y\\
  &=\gamma_{\bm l}\int_\Omega g(y)\mathrm{e}^{-\mathrm{i}\left(\frac{2\pi}{a}|\bm l|\right)
 \frac{\bm l}{|\bm l|}\cdot y}\mathrm{d}y\\
 &=\gamma_{\bm l}\int_\Omega g(y)\mathrm{e}^{-\mathrm{i}\frac{2\pi}{a}{\bm l}\cdot y}\mathrm{d}y\\
 &=\gamma_{\bm l}\int_\Omega g(y)\overline{\phi_{\bm l}(y)}\mathrm{d}y\\
 &=a^2\gamma_{\bm l}\widehat{g}_{\bm l},
\end{align*}
where and in what follows $\gamma_{\bm l}=\mathrm{e}^{\mathrm{i}\pi/4}/\sqrt{8\pi k_{\bm l}}$. The above deduction leads to the computational formula for the Fourier coefficients $\widehat{g}_{\bm l}$, i.e.,
\begin{align}\label{eq: gl}
\widehat{g}_{\bm l}=\frac{1}{a^2\gamma_{\bm l}}\mathbf{E}\left[u^\infty(\widehat{x}_{\bm l};k_{\bm l})\right],\quad {\bm l}\ne{\bm 0}.
\end{align}

For  ${\bm l}={\bm 0},$ we find that 
\begin{align*}
  \mathbf{E}\left[u^\infty(\widehat{x}_{\bm0};k_{\bm0})\right]&=\gamma_{{\bm l}_0}
  \int_\Omega g(y)\overline{\phi_{{\bm l}_0}(y)}\mathrm{d}y\\
  &=\gamma_{{\bm l}_0}\int_\Omega \left(\widehat{g}_{\bm0}+\sum_{\bm{l}\in\mathbb{Z}^2\backslash\{\bm 0\}}\widehat{g}_{\boldsymbol l}\phi_{\boldsymbol l}(y)\right)\overline{\phi_{{\bm l}_0}(y)}\mathrm{d}y\\
  &=\gamma_{{\bm l}_0}a^2\frac{\sin\lambda\pi}{\lambda\pi}\widehat{g}_{\bm0}
  +\gamma_{{\bm l}_0}\sum_{\bm{l}\in\mathbb{Z}^2\backslash\{\bm 0\}}\widehat{g}_{\boldsymbol l}
  \int_\Omega\phi_{\boldsymbol l}(y)\overline{\phi_{{\bm l}_0}(y)}\mathrm{d}y,
\end{align*}
which further gives 
\begin{align}\label{eq: g0}
\widehat{g}_{\bm 0}=\frac{\lambda\pi}{a^2\sin\lambda\pi}\left(
\frac{1}{\gamma_{{\bm l}_0}}\mathbf{E}\left[u^\infty(\widehat{x}_{\bm0};k_{\bm0})\right]-\sum_{{\bm l}\in\mathbb{Z}^2\backslash\{\boldsymbol{0}\}}\widehat{g}_{\bm l}\int_\Omega\phi_{\bm l}(y)\overline{\phi_{{\bm l}_0}(y)}\mathrm{d}y
\right).
\end{align}
 \end{proof}
 
 Based on the above theorem, we approximate the deterministic source $g$ by the following truncated Fourier series
 \begin{align}\label{eq: gN}
 g_N(x)=\widehat{g}_0+\sum_{1\le|\bm l|_\infty\le N}\widehat{g}_{\bm l}\phi_{\bm l}(x).
  \end{align}
 
 \subsection{Reconstruct the variance function $\sigma^2$}
 In this subsection, we are concerned with the inverse problem to reconstruct the variance function $\sigma^2(x)$ from the measurement $\mathbf{C}\left[u^\infty(\widehat{x};k_0+\tau),u^\infty(\widehat{x}; k_0)\right]$. With a slight difference from the admissible wavenumber to reconstruct the mean function, the admissible wavenumber for the variance reconstruction should be adjusted correspondingly as the following:
 \begin{definition}[Admissible wavenumber to reconstruct the variance function]
 \label{def: admissible_variance}
Take $k_0>0$ and select the admissible wavenumber as follows:
\begin{align}\label{eq: admissible_variance}
\tau_{\bm l}:=\frac{2\pi}{a}|\bm l|,\quad {\bm l}\in\mathbb{Z}^2.
\end{align}
Correspondingly, the observation is defined by 
\begin{align}\label{eq: xhat_tau}
\widehat{x}_{\bm l}:=
\left\{
\begin{aligned}
&\frac{\bm l}{|\bm l|},&&{\bm l}\ne{\bm0},\\
&{\bm a},&&{\bm l}={\bm0},
\end{aligned}
\right.
\end{align}
with ${\bm  a}=(a_1, a_2)\in\mathbb{S}$ being some unit vector chosen flexibly.
 \end{definition}
 
 Based on Definition \ref{def: admissible_variance}, we can prove that 
 \begin{theorem}
   For $\bm{l}\in\mathbb{Z}^2,$ the Fourier coefficients $\{\widehat{\sigma}_{\bm l}\}$ of $\sigma^2$  in \eqref{eq: sigma} can be determined by $\mathbf{C}\left[u^\infty(\widehat{x}; k_0+\tau_{\bm{l}}), u^\infty(\widehat{x}; k_0)\right]$, with $k_0>0$ being a small wavenumber.
 \end{theorem}
 
 \begin{proof}
 We first recall the crucial property that 
 \begin{align}\label{eq: property}
 \mathbf{E}\left[\sigma(y)\dot{W}_y\overline{\sigma(z)\dot{W}_z}\right]=\sigma(y)\sigma(z)\delta(y-z),
 \end{align}
 where 
$$
 \delta(x)=
 \begin{cases}
 1, & x=0,\\
 0, & x\ne0.
 \end{cases}
$$

For notational convenience, we denote 
$$
   U(\widehat{x};k)=\frac{\sqrt{8\pi k}}{\mathrm{e}^{\mathrm{i}\pi/4}}u^\infty(\widehat{x}; k).
$$
From \eqref{eq: acoustic_far} and \eqref{eq: Efar}, it holds that 
\begin{align*}
  U(\widehat{x}; k)-\mathbf{E}\left[U(\widehat{x};k)\right]=\int_\Omega\sigma(y)\dot{W}_y\mathrm{e}^{-\mathrm{i}k\widehat{x}\cdot y}\mathrm{d}y.
\end{align*}
Furthermore, from the relation between the covariance and mean, we derive that for small wavenumber $k_0>0$ and ${\bm l}\in\mathbb{Z}^2$
\begin{align*}
  &\quad\frac{1}{a^2}\mathbf{C}\left[U(\widehat{x}_{\bm l};k_0+\tau_{\bm l}),U(\widehat{x}_{\bm l};k_0)\right]\\
  &=\frac{1}{a^2}\mathbf{E}\left[\left(U(\widehat{x}_{\bm l};k_0+\tau_{\bm l})-\mathbf{E}\left[U(\widehat{x}_{\bm l};k_0+\tau_{\bm l})\right]\right)
  \overline{\left(U(\widehat{x}_{\bm l};k_0)-\mathbf{E}\left[U(\widehat{x}_{\bm l};k_0)\right]\right)}\right]\\
  &=\frac{1}{a^2} \int_\Omega\int_\Omega\mathrm{e}^{-\mathrm{i}(k_0+\tau_{\bm l})\widehat{x}_{\bm l}\cdot y}
  \mathrm{e}^{\mathrm{i}k_0\widehat{x}_{\bm l}\cdot z}\mathbf{E}\left[\sigma(y)\dot{W}_y\overline{\sigma(z)\dot{W}_z}\right]\mathrm{d}y\mathrm{d}z.
\end{align*}
This together with \eqref{eq: property} leads to 
\begin{equation}\label{eq: sigmal}
  \frac{1}{a^2}\mathbf{C}\left[U(\widehat{x}_{\bm l}; k_0+\tau_{\bm l}), U(\widehat{x}_{\bm l}; k_0)\right]
  =\frac{1}{a^2}\int_\Omega\sigma^2(y)\mathrm{e}^{-\mathrm{i}\tau_{\bm l}\widehat{x}_{\bm l}\cdot y}\mathrm{d}y
 =\widehat{\sigma}_{\bm l}.
\end{equation}
 
Especially, for ${\bm l}={\bm0},$ the Fourier coefficient is defined by
$$
    \widehat{\sigma}_{\bm 0}=\frac{1}{a^2}\mathbf{C}[U(\widehat{x}_0; k_0),U(\widehat{x}_0; k_0)]
                                  =\frac{1}{a^2}\int_\Omega\sigma^2(y)\mathrm{d}y,
$$
 which is independent of the observation direction $\widehat{x}_0$ essentially. Thus, the covariance $\mathbf{C}[U(\widehat{x}_0; k_0), U(\widehat{x}_0; k_0)]$ can be computed by selecting some observation direction $\widehat{x}_0=\bm{a}\in\mathbb{S}$ and it is not necessary to relate $\widehat{x}_0$ with the Fourier index ${\bf 0}=(0,0).$
 \end{proof}
 
 The above theorem showcases an approximation to the variance function $\sigma^2,$ which can be defined via
 \begin{align}\label{eq: sigmaN}
 \sigma_N^2(x)=\sum_{|\bm{l}|_\infty\le N}\widehat{\sigma}_{\boldsymbol l}\phi_{\boldsymbol l}(x).
 \end{align}

At the end of this subsection, we would like to compare the reconstruction schemes developed for the mean and the variance function. A main difference is that when ${\bm l}={\bm 0},$ we should select a small wavenumber and utilize the corresponding radiated field to compute the corresponding Fourier coefficient. Conversely, this step is unnecessary for reconstructing the variance function $\sigma^2(x).$ This improvement enables us to achieve a better reconstruction, as illustrated in the next subsection.

 \subsection{Numerical simulation}

This section aims to verify the acoustic-version Fourier method for determining the stochastic source. The radiated data is obtained by the numerical solution of the stochastic Helmholtz equation to avoid the inverse crime. Specifically, for each realization, we compute the far-field of the radiated field by the direct integration, whose explicit computational formula is given by \eqref{eq: acoustic_far}, and the integration domain is chosen to be $\Omega=[-0.5,0.5]\times[-0.5,0.5]$. After the realizations are done, we take the average of the solutions to approximate the corresponding scattering data, i.e., the mean or the covariance. The total number of realizations is $10^6.$ In general, more realizations would result in enhanced accuracy in the data.
 
To test the stability of the proposed method, we add some noise to the far field data, i.e.,
 $$
 u^{\infty,\delta}:=u^\infty+\delta r_1|u^\infty|\mathrm{e}^{\mathrm{i}\pi r_2},
 $$
 with $r_\ell,\ell=1,2,$ being two uniformly distributed random numbers that range from $-1$ from 1, $\delta\in[0,1)$ is the noise level. 
 
 In this section, the truncation $N$ in \eqref{eq: gN} and \eqref{eq: sigmaN}  is chosen to be
 $$
 N = 2\left[\delta^{-1/2}\right],
 $$
 here, $[X]$ denotes the largest integer that is smaller than $X+1.$
  Then, the admissible wavenumbers defined by \Cref{def: admissible_mean} and \Cref{def: admissible_variance} are respectively given by
   \begin{align*}
   k_{\boldsymbol l}:=\left\{
   \begin{aligned}
   &{2\pi}|\bm{l}|,&&0<|\bm{l}|_\infty\le N\\
   &{2\pi}\times 10^{-3},&&\bm{l}=\boldsymbol{0}
   \end{aligned}
   \right.,\quad 
\tau_{\bm l}:=\frac{2\pi}{a}|\bm l|,\quad |\bm{l}|_\infty\le N.
   \end{align*}
   
   Once the admissible wavenumbers  are determined, the synthesis statistical measurements (mean and variance) can be given by
   \begin{align*}
     \left\{
     \mathbf{E}\left[u^{\infty,\delta}(\widehat{x}_{\bm l};k_{\bm l})\right],\ \
     \mathbf{C}\left[u^{\infty,\delta}(\widehat{x}_{\bm l};k_0+\tau_{\bm l}),u^{\infty,\delta}(\widehat{x}_{\bm l};k_0)\right]
     \right\},
   \end{align*}
 and $k_0>0$ is the small wavenumber that will be specified later.
 
 Given the measured data $\left\{\mathbf{E}\left[u^{\infty,\delta}(\widehat{x}_{\bm l};k_{\bm l})\right]; \mathbf{C}\left[u^{\infty,\delta}(\widehat{x}_{\bm l};k_0+\tau_{\bm l}),u^{\infty,\delta}(\widehat{x}_{\bm l};k_0)\right]\right\}$, we compute the Fourier coefficients $\widehat{g}_{\bm l}^\delta,1\le|\bm l|_\infty\le N$ and $\widehat{\sigma}_{\bm l}^\delta$ via \eqref{eq: gl} and \eqref{eq: sigmal}, respectively. For $\bm l=\bm 0,$ $\widehat{g}^\delta_{\bm0}$ is given by \eqref{eq: g0}.
 
 For quantitatively characterizing the reconstruction, we introduce a $401\times401$ grid of uniformly spaced points $x_m\in V_0,m=1,\cdots,401^2$, and compute the relative $L^2$  error as
\begin{align*}
  &\frac{\|w^\delta_N-w\|_0}{\|w\|_0}=\frac{\left(\sum_{m=1}^{401^2}\left|w_N^\delta(x_m)-w(x_m)\right|^2\right)^{1/2}}{\left(\sum_{m=1}^{401^2}\left|w(x_m)\right|^2\right)^{1/2}}\\
  &\frac{\|w^\delta_N-w\|_1}{\|w\|_1}=\frac{\left(\sum_{m=1}^{401^2}\left(\left|\nabla w_N^\delta(x_m)-\nabla w(x_m)\right|^2+\left|w_N^\delta(x_m)-w(x_m)\right|^2\right)\right)^{1/2}}
    {\left(\sum_{m=1}^{401^2}\left|\nabla w(x_m)\right|^2+\left|w(x_m)\right|^2\right)^{1/2}}
\end{align*}
with $w$ being $g$ or $\sigma^2,$ depending on the statistical information to be determined, and $w_N^\delta$ is the reconstruction utilizing the noisy measurements.
 
Based on the above preparation, we are now to deliver a numerical example in which the mean function $g$ is given by
$$
   g(x_1, x_2)=\mathrm{e}^{-200((x-0.01)^2+(y-0.12)^2)}-100(y^2-x^2)\mathrm{e}^{-90(x^2+y^2)},
$$
and the standard derivation function $\sigma$ is given by
$$
    \sigma(x_1, x_2)=\frac{1}{2}g(x_1, x_2).
$$

We refer to \Cref{fig: S1} for the surface plots of the exact mean function $g$ and the variance function $\sigma^2$ in $\Omega$.

\begin{figure}
   \centering
   \subfigure{\includegraphics[width=0.45\linewidth]{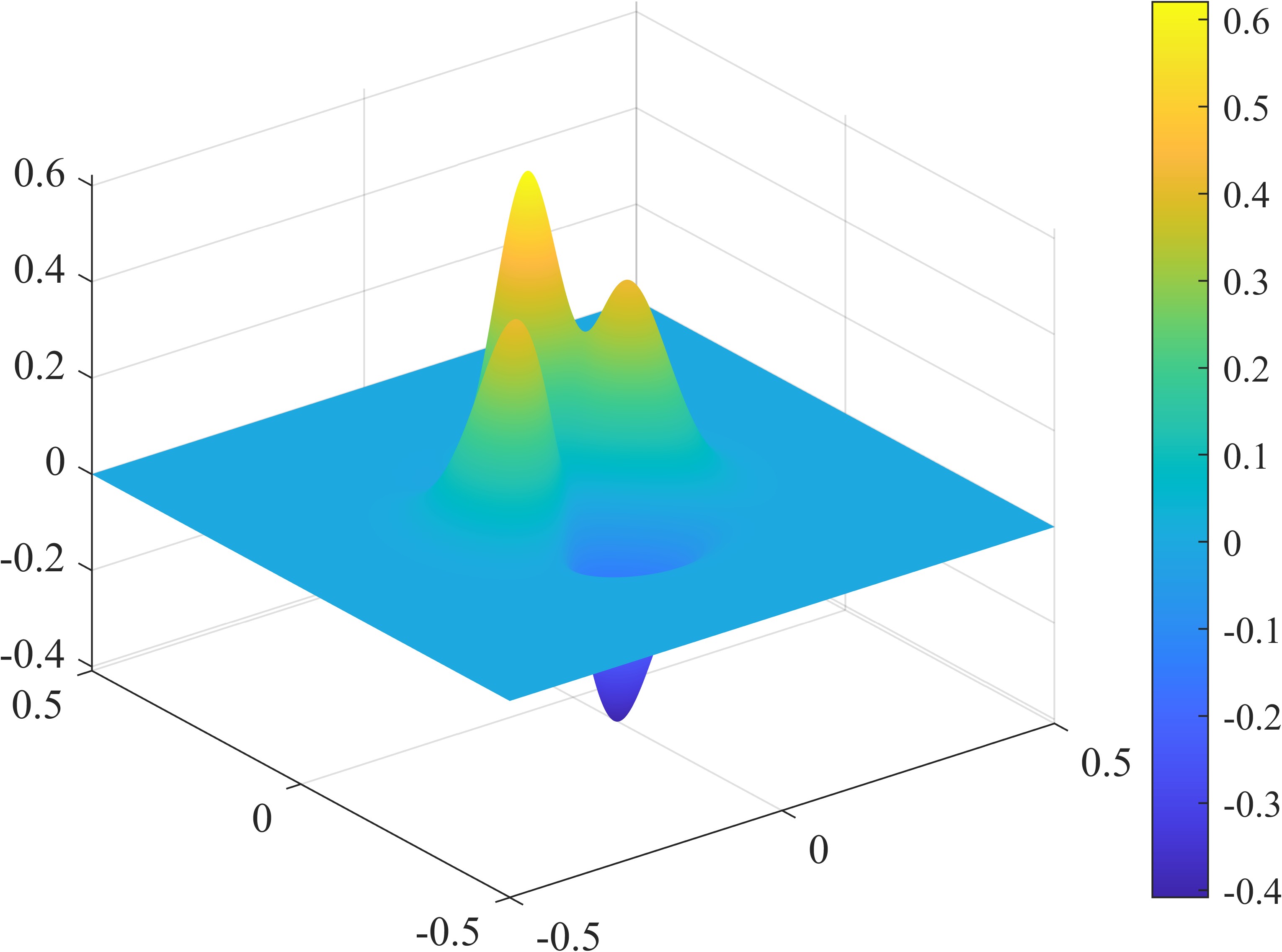}}\qquad
   \subfigure{\includegraphics[width=0.45\linewidth]{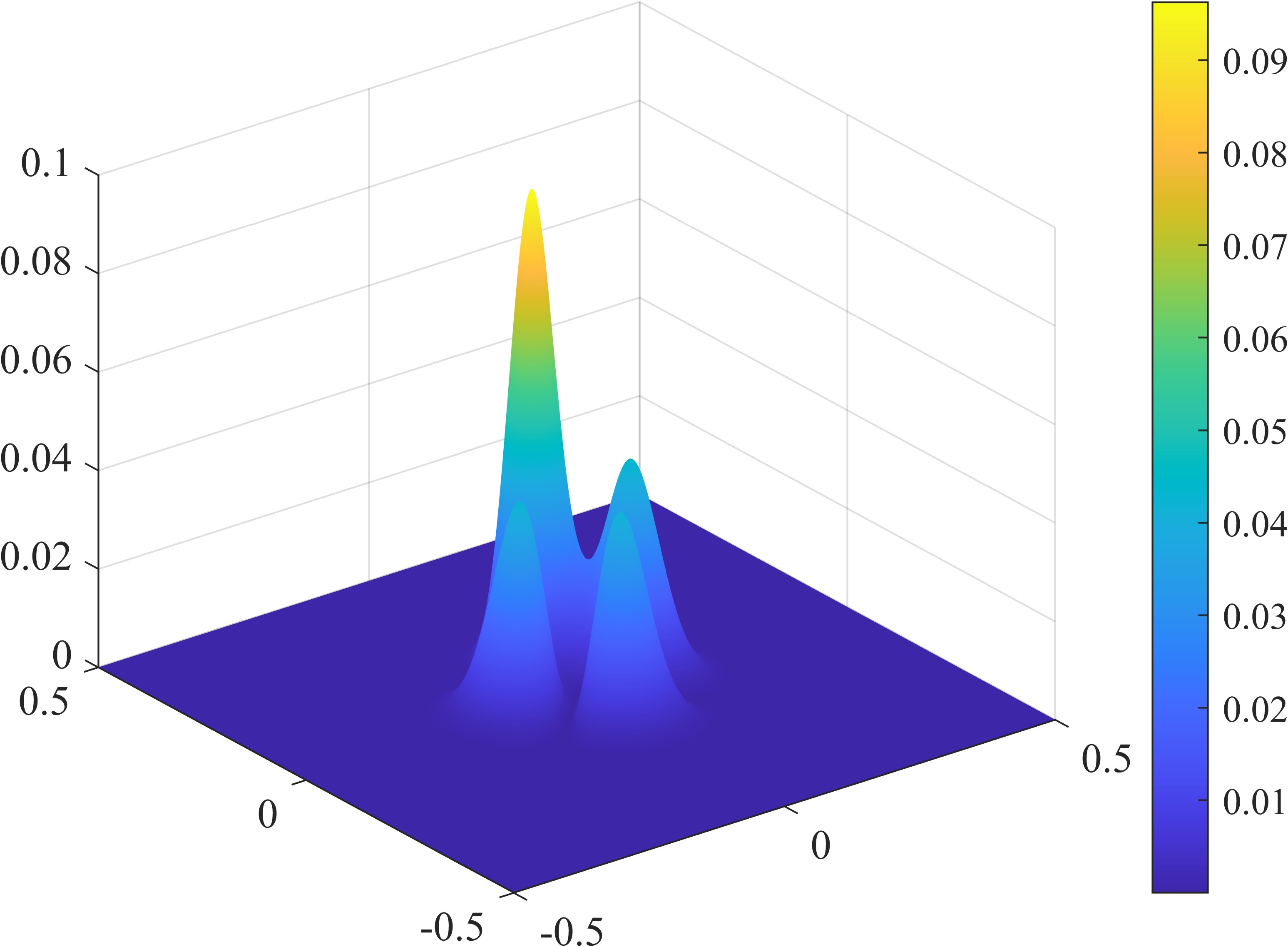}}
  \caption{The exact source function. (a) mean function $g$ (b) the variance function $\sigma^2.$}\label{fig: S1}
\end{figure}

Taking $k_0=1, \, \delta=5\%,$ we display the reconstructed mean $g_N$ and variance $\sigma_N^2$ in \Cref{fig: S1N}. Comparing \Cref{fig: S1} and \Cref{fig: S1N}, we can observe that both the acoustic mean function $g$ and the variance function $\sigma^2$ are well-reconstructed.
For quantitative comparison, we further list in \Cref{tab: error} the relative errors of the reconstruction for different noise levels. 
We find from \Cref{tab: error} that our method can realize a quantitative reconstruction no matter the mean function or the covariance function, which further illustrates that our method is effective in reconstructing the mean function and the covariance function. For an intuitive visualization, we refer to  \Cref{fig: S1N10} for the reconstruction of the mean and the covariance. Comparing \Cref{fig: S1N}, \Cref{fig: S1N10} and \Cref{tab: error}, we find that though the error for $\delta=10\%$ may be relatively large compared with that for $\delta=5\%,$ the reconstruction is still satisfactory and this further illustrates the capability of robust reconstruction by the proposed method.
 
\begin{figure}
  \centering
   \subfigure{\includegraphics[width=0.45\linewidth]{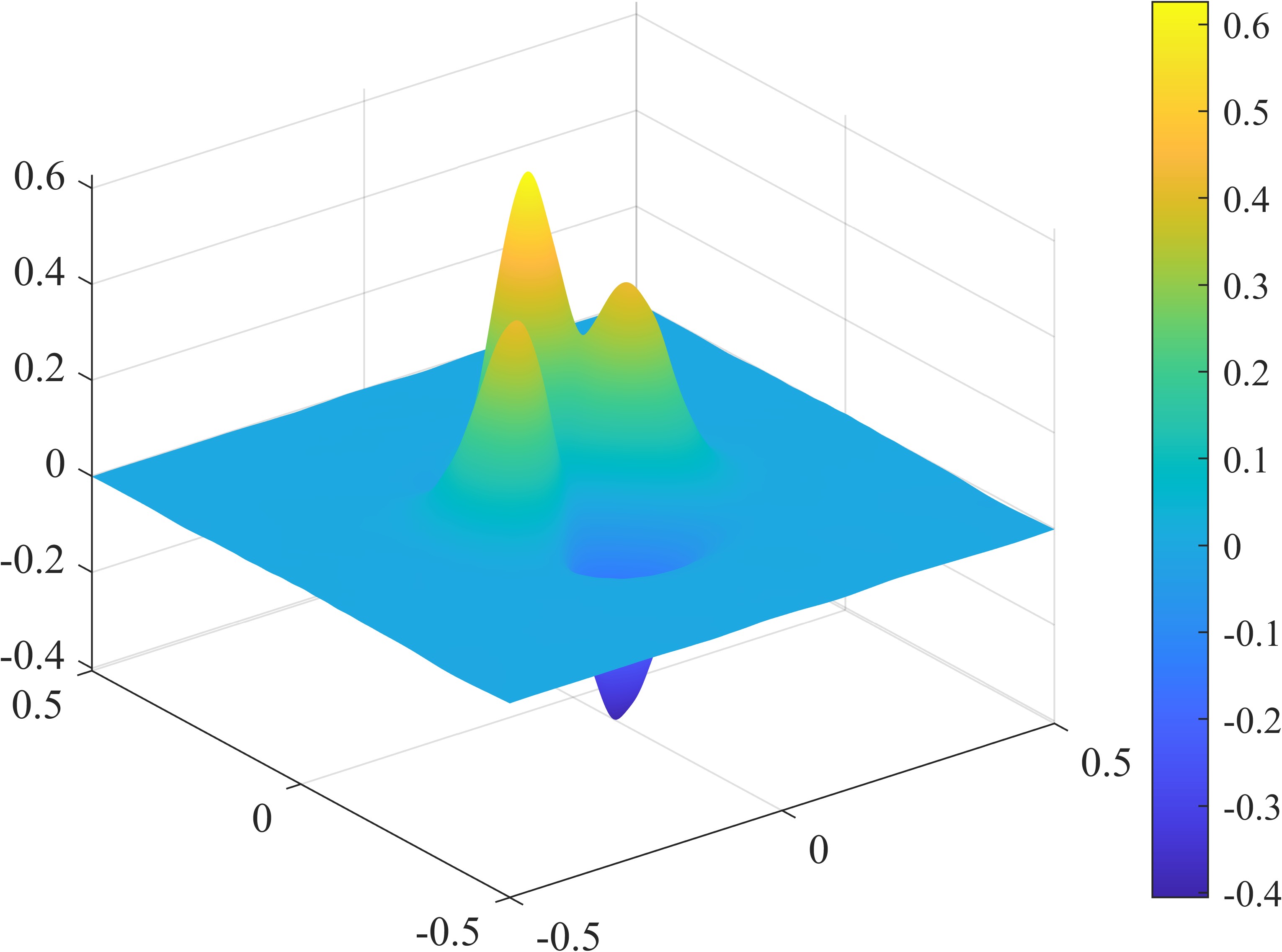}}\qquad
   \subfigure{\includegraphics[width=0.45\linewidth]{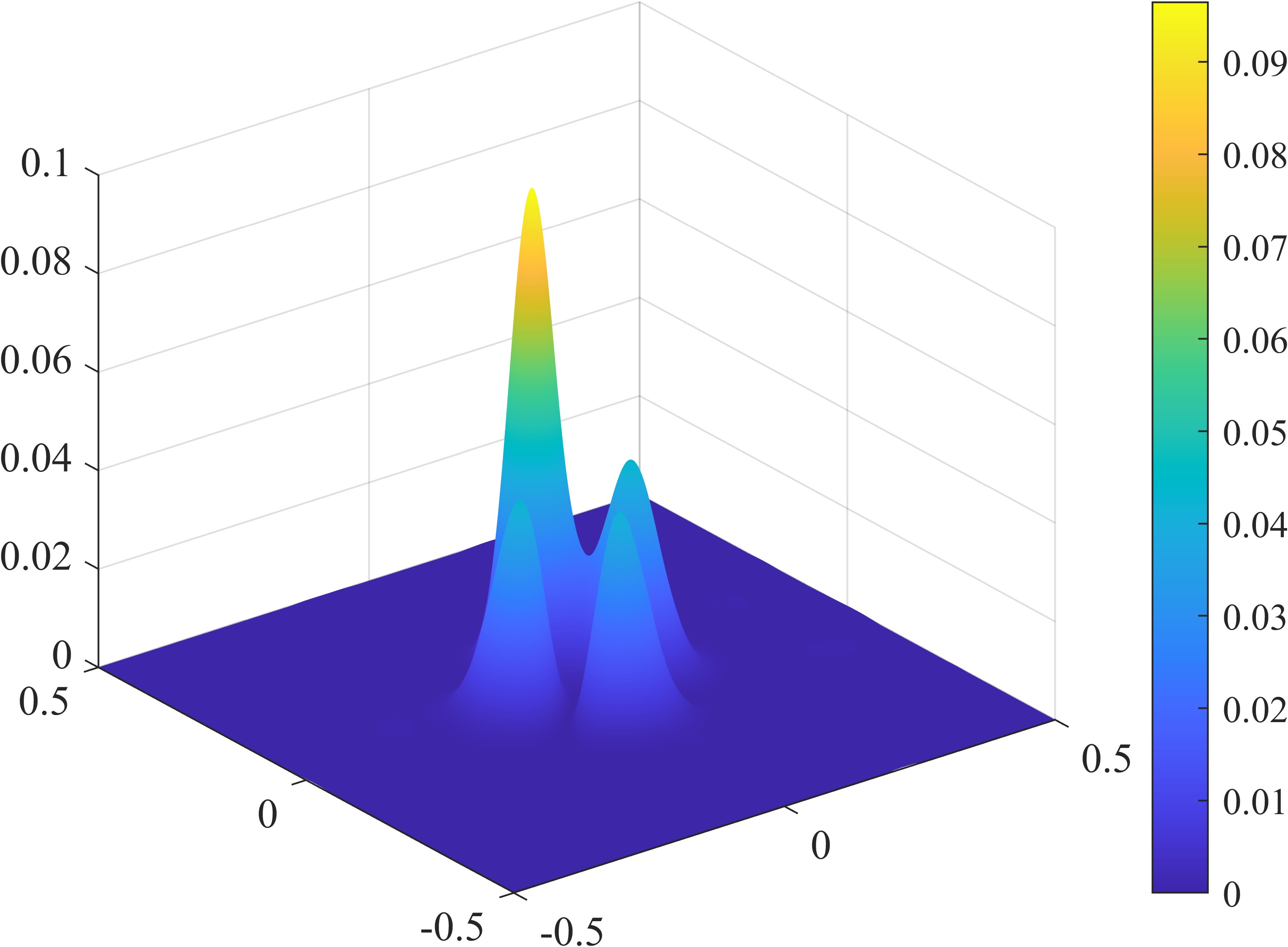}}
  \caption{The reconstructed source function for $\delta = 5\%$. (a) mean function $g_N$ (b) the variance function $\sigma_N^2.$}\label{fig: S1N}
\end{figure}

\begin{table}
  \centering
  \begin{tabular}{cccccc}
    \toprule
    $\delta$ & $0.5\%$ & $1\%$ & $5\%$ &$10\%$ \\
    \midrule
    \vspace{0.2cm}
    ${\|g^\delta_N-g\|_0}/{\|g\|_0}$&$1.84\%$&$1.92\%$&$3.57\%$&$6.38\%$\\
    \vspace{0.2cm}
    ${\|g^\delta_N-g\|_1}/{\|g\|_1}$&$3.55\%$&$3.59\%$&$4.68\%$&$7.06\%$\\
    \vspace{0.2cm}
    ${\|\sigma^{2,\delta}_N-\sigma^2\|_0}/{\|\sigma^2\|_0}$&$0.97\%$&$1.09\%$&$3.09\%$&$5.97\%$\\
    ${\|\sigma^{2,\delta}_N-\sigma^2\|_1}/{\|\sigma^2\|_1}$ & $1.71\%$ & $1.78\%$ & $3.37\%$ & $6.09\%$ \\
    \bottomrule
  \end{tabular}
  \caption{The relative errors of the reconstruction of $g_N$ and $\sigma_N^2$ for different noise levels.}\label{tab: error}
\end{table}

\begin{figure}
    \centering
    \subfigure{\includegraphics[width=0.4\linewidth]{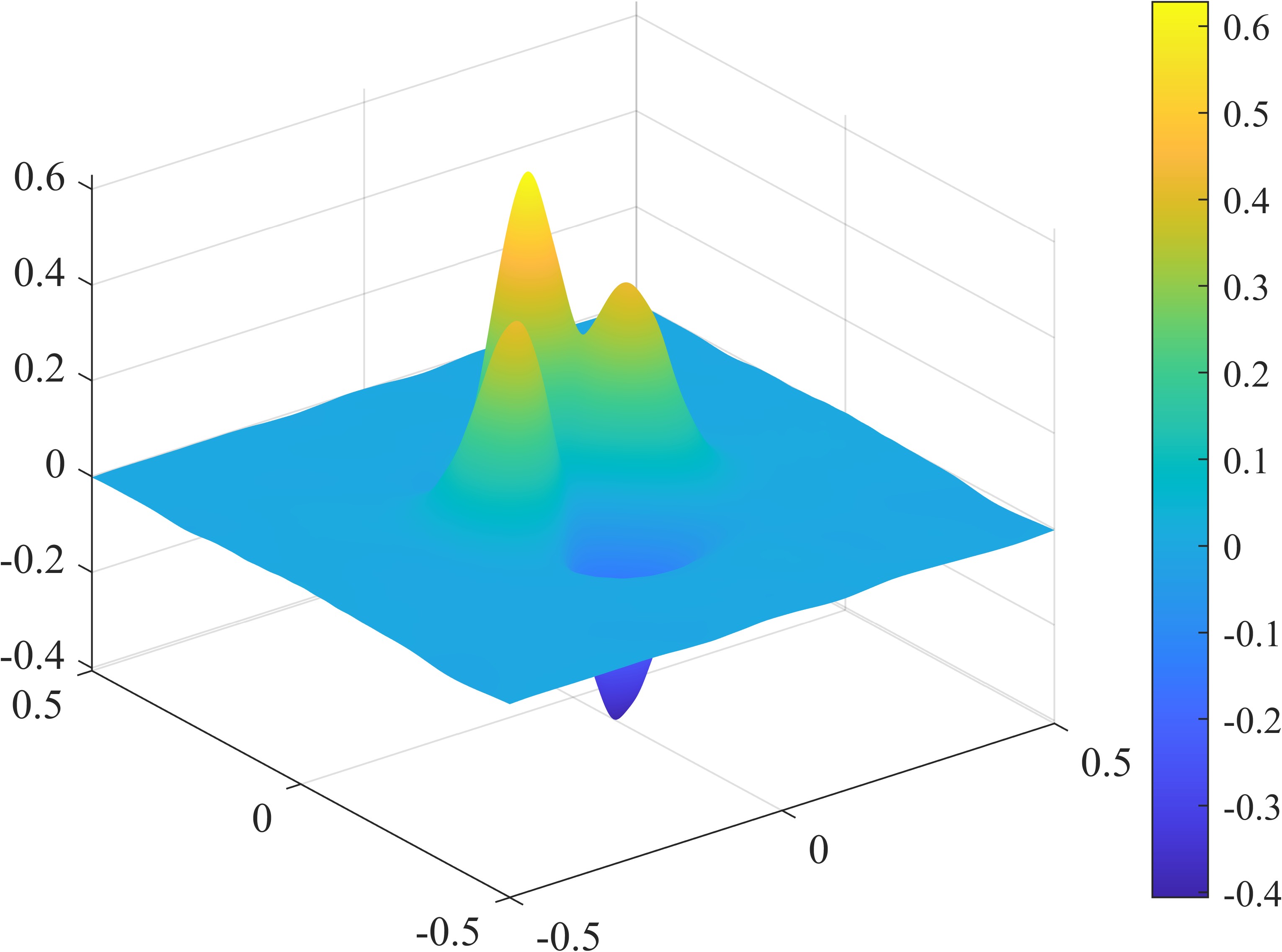}}\qquad
    \subfigure{\includegraphics[width=0.4\linewidth]{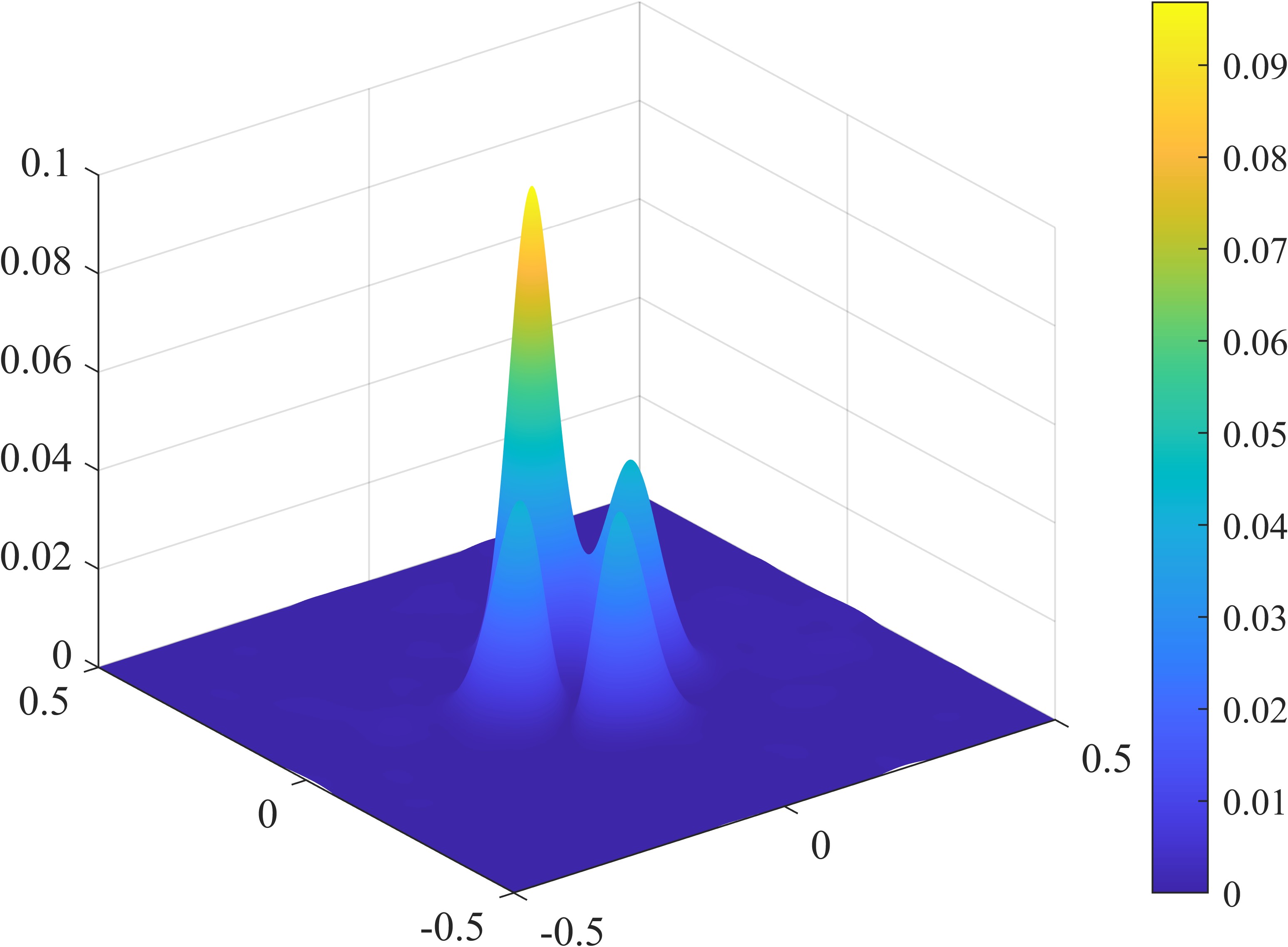}}
    \caption{The reconstructed source function for $\delta=10\%$. (a) the mean $g_N$ (b) the variance $\sigma_N^2.$}\label{fig: S1N10}
\end{figure}


\section{Elastic waves}
 
In this section, we are concerned with the inverse random source problems for elastic waves in an isotropic homogeneous medium. We shall start this section with a brief mathematical description of the multifrequency inverse elastic stochastic source problem under consideration and then develop the corresponding reconstruction scheme based on the Fourier expansion. Several numerical experiments will be presented to demonstrate the feasibility of the proposed method.
 \subsection{Problem formulation}
 The displacement of elastic wave $\bm{u}\in\mathbb{C}^2(\mathbb{R}^2)$ is modelled by the following stochastic Navier equation:
 \begin{align}\label{eq: Navier}
 \mu\Delta \bm{u}+(\lambda+\mu)\nabla\nabla\cdot\bm{u}+\omega^2\bm{u}=-\boldsymbol{F} \quad \text{in}\ \mathbb{R}^2,
 \end{align}
 where $\omega>0$ is the angular frequency, $\lambda$ and $\mu$ are the Lam\'{e} constants satisfying $\mu>0,\lambda+\mu>0.$ The source $\boldsymbol{F}=(F_1, F_2)^\top$ is assumed to be a random function driven by an additive white noise, which takes the form
 \begin{align}\label{eq: F}
 \boldsymbol{F}=\boldsymbol{g}+\boldsymbol{\sigma}\dot{W}_x,
 \end{align}
 where $x=(x_1,x_2)^\top\in\mathbb{R}^2,$ $\boldsymbol{g}=(g_1,g_2)^\top$ is a deterministic real-valued vector function, the diagonal matrix function $\boldsymbol{\sigma}=\text{diag}(\sigma_1,\sigma_2)$ is deterministic, with whose element $\sigma_\ell\ge0,\,\ell=1,2.$ Throughout this paper, we assume that the functions $g_\ell,\sigma_\ell,\,\ell=1,2$ are independent of angular frequency $\omega,$ and have compact supports that are contained in the squared domain $\Omega\subset\mathbb{R}^2.$ 
 In \eqref{eq: f}, $W(x)=(W_1(x), W_2(x))^\top$ is a two-dimensional Brownian sheet of two parameters with $W_\ell(x),\ell=1,2$ being independent one-dimensional two-parameter Brownian sheets in the probability space. Further, the white noise $\dot{W}_x$ is the derivative of the Brownian sheet $W(x).$ 
Within the random source model, the functions $\boldsymbol{g}$ and $\boldsymbol{\sigma}$ can be viewed as the mean and the standard derivation of $\boldsymbol{F},$ respectively. Correspondingly, $\boldsymbol{\sigma}^2=\text{diag}(\sigma_1^2,\sigma_2^2)$  is the variance of $\boldsymbol{F}.$
 
For any solution $\bm{u}$ of \eqref{eq: Navier}, the Helmholtz decomposition reads 
$$
  \bm{u}=\bm{u}_p+\bm{u}_s,
$$
with 
\begin{align*}
\bm{u}_p=-\frac{1}{k_p^2}\nabla\nabla\cdot\bm{u},\quad
\bm{u}_s=\frac{1}{k_s^2}\textbf{curl}\text{curl}\bm{u},
\end{align*}
being the compressional and shear parts of the displacement $\bm{u},$ respectively. Here, $k_p$ and $k_s$ are the compressional wavenumber and the shear wavenumber, respectively, which are given by
$$
  k_p=\frac{\omega}{c_p},\quad k_s=\frac{\omega}{c_s}.
$$
with 
$
c_p=\sqrt{\lambda+2\mu},\ \ c_s=\sqrt{\mu}.
$
In addition, $\bm{u}_p$ and $\bm{u}_s$ are supposed to satisfy the following Sommerfeld radiation condition
\begin{align}\label{eq: Sommerfeldelastic}
\lim\limits_{r=|x|\to\infty}\sqrt{r}\left(\frac{\partial \bm{u}_\xi}{\partial r}-\mathrm{i}k_\xi\bm{u}_\xi\right)=0,\quad \xi\in\{p,s\},
\end{align}
uniformly in all directions $\widehat{x}=x/|x|$.

Denote by $\mathbb{G}(x,y;\omega)\in\mathbb{C}^{2\times2}$ the Green tensor to the Navier equation, which takes the form 
\begin{align*}
  \mathbb{G}(x, y; \omega) =\frac{1}{\mu}G_{k_s}(x,y)\mathbb{I}+\frac{1}{\omega^2}\nabla_x\nabla_x^\top
  \left(G_{k_s}(x,y)-G_{k_p}(x,y)
  \right),
\end{align*}
where $\mathbb{I}\in\mathbb{R}^{2\times 2}$ is the identity matrix and $G_{k_\xi}(x,y)$ is the fundamental solution to the Helmholtz equation of form \eqref{eq: Phi} with $k=k_\xi,\,\xi\in\{p,s\}.$ Then the unique solution to \eqref{eq: Navier} and \eqref{eq: Sommerfeldelastic} is given by 
\begin{align}\label{eq: solution_elastic}
\bm{u}(x;\omega)=\int_\Omega\mathbb{G}(x,y;\omega) \bm{F}(y) \mathrm{d}y.
\end{align}

It has been shown in \cite{Arens}  that the radiating solution $\bm{u}$ to the Navier equation admits the following asymptotic behavior 
\begin{align}\label{eq: asymptotic}
\bm{u}(x;\omega)=\frac{\mathrm{e}^{\mathrm{i}k_p|x|}}{\sqrt{|x|}}\bm{u}_p^\infty(\widehat{x};\omega)
+\frac{\mathrm{e}^{\mathrm{i}k_s|x|}}{\sqrt{|x|}}\bm{u}_s^\infty(\widehat{x};\omega)+\mathcal{O}\left(|x|^{-3/2}\right),\quad |x|\to\infty,
\end{align}
which holds uniformly in all directions $\widehat{x}=x/|x|.$ In \eqref{eq: asymptotic}, $\bm{u}_p^\infty(\widehat{x};\omega)$ and $\bm{u}_s^\infty(\widehat{x};\omega)$ are known as the compressional and shear far-field patterns of the scattered field $\bm{u}$, respectively, and 
\begin{align}\label{eq: up_far}
&\bm{u}_p^\infty(\widehat{x};\omega)=\gamma_{p}\frac{k_p^2}{\omega^2}\int_\Omega\widehat{x}\widehat{x}^\top\mathrm{e}^{-\mathrm{i}k_p\widehat{x}\cdot y}\boldsymbol{F}(y)\mathrm{d}y,\\\label{eq: us_far}
&\bm{u}_s^\infty(\widehat{x};\omega)=\gamma_{s}\frac{k_s^2}{\omega^2}\int_\Omega\left(\mathbb{I}-\widehat{x}\widehat{x}^\top\right)\mathrm{e}^{-\mathrm{i}k_s\widehat{x}\cdot y}\boldsymbol{F}(y)\mathrm{d}y,
\end{align}
with $\gamma_\xi=\frac{\mathrm{e}^{\mathrm{i}\pi/4}}{\sqrt{8\pi k_\xi}},\,\xi\in\{p,s\}.$

Taking the expectation on both side of \eqref{eq: up_far}--\eqref{eq: us_far}  gives
\begin{align}
\mathbf{E}\left[\bm{u}_p^\infty(\widehat{x};\omega)\right]&=\gamma_{p}\frac{k_p^2}{\omega^2}\int_\Omega\widehat{x}\widehat{x}^\top\mathrm{e}^{-\mathrm{i}k_p\widehat{x}\cdot y}\boldsymbol{g}(y)\mathrm{d}y,\label{eq: up_far_E}\\
\mathbf{E}\left[\bm{u}_s^\infty(\widehat{x};\omega)\right]&=\gamma_{s}\frac{k_s^2}{\omega^2}\int_\Omega\left(\mathbb{I}-\widehat{x}\widehat{x}^\top\right)\mathrm{e}^{-\mathrm{i}k_s\widehat{x}\cdot y}\boldsymbol{g}(y)\mathrm{d}y,\label{eq: us_far_E}
\end{align}
where \cite[Proposition A.3]{BCL17} has been taken into account. Further, one can easily deduce that 
\begin{align}\label{eq: up_Eup}
   {\bm u}_p^\infty(\widehat{x};\omega)-\mathbf{E}\left[{\bm u}_p^\infty(\widehat{x};\omega)\right]&=\gamma_{p}\frac{k_p^2}{\omega^2}\int_\Omega\widehat{x}\widehat{x}^\top
   \mathrm{e}^{-\mathrm{i}k_p\widehat{x}\cdot y}\boldsymbol{\sigma}(y)\mathrm{d}W_y,
   \\\label{eq: us_Eus}
   {\bm u}_s^\infty(\widehat{x};\omega)-\mathbf{E}\left[{\bm u}_s^\infty(\widehat{x};\omega)\right]
   &=\gamma_{s}\frac{k_s^2}{\omega^2}\int_\Omega\left(\mathbb{I}-\widehat{x}\widehat{x}^\top\right)\mathrm{e}^{-\mathrm{i}k_s\widehat{x}\cdot y}\boldsymbol{\sigma}(y)\mathrm{d}W_y.
 \end{align}

Similar to the ISP for the acoustic wave, we take some small $\omega_0>0$ and compute the covariance $\mathbf{C}\left[\bm{u}_\alpha^\infty(\widehat{x};\omega_0+\tau),\bm{u}_\beta^\infty(\widehat{x};\omega_0+\tau)\right]$
 for each $\tau>0$ with $(\alpha,\beta)\in\{(p,p),(p,s),(s,p),(s,s)\}$.
 
We are now in a position to formulate the following ISP for the elastic wave:
 \begin{problem}
   Choose two sets of angular frequencies $\{\omega\}$ and $\{\tau\},$ and take some small frequency $\omega_0>0.$ Then the ISP for the elastic wave is to determine the mean $\boldsymbol{g}$ and the variance $\boldsymbol{\sigma}^2$ from the multi-frequency far-field statistical measurements
   \begin{align*}
    \left\{
     \mathbf{E}\left[\bm{u}_\xi^\infty(\widehat{x};c_\xi\omega)\right]: \xi\in\{p,s\}
     \right\},
   \end{align*}
   and 
   \begin{align*} \left\{  \mathbf{C}\left[\bm{u}_\alpha^\infty(\widehat{x};c_\alpha(\omega_0+\tau)),\bm{u}_\beta^\infty(\widehat{x};c_\beta(\omega_0+\tau))\right]
:(\alpha,\beta)\in\{(p,p),(p,s),(s,p),(s,s)\}
\right\}
,
   \end{align*}
   respectively.
 \end{problem}
 
 For convenience, we introduce several notations concerning the vector functions. Again, we assume that the compact supports of $\bm{g}$ and $\bm{\sigma}$ are contained in
 \[
 \Omega=\left(-\frac{a}{2},\frac{a}{2}\right)\times\left(-\frac{a}{2},\frac{a}{2}\right).
 \]
 For any function $\boldsymbol{v}=(v^{(1)},v^{(2)})^\top\in (L^2(\Omega))^2$, the Fourier series expansion reads
 \[
 \boldsymbol{v}(x)=\left(
 \sum_{\bm{l}\in\mathbb{Z}^2}v_{\bm{l}}^{(1)}\phi_{\bm{l}}(x),\,
 \sum_{\bm{l}\in\mathbb{Z}^2}v_{\bm{l}}^{(2)}\phi_{\bm{l}}(x)
 \right)^\top,
 \]
 where $\phi_{\bm{l}}(x)$ are  the Fourier basis functions defined through \eqref{eq: basis}, and the Fourier coefficients are given by 
 \begin{align*}
   v_{\bm{l}}^{(j)}=\frac{1}{a^2}\int_\Omega v^{(j)}(x)\overline{\phi_{\bm{l}}(x)}\mathrm{d}s,
   \quad j=1,2.
 \end{align*}
Denote 
 \[
 \boldsymbol{v}_{\bm{l}}=\left( v_{\bm{l}}^{(1)}, v_{\bm{l}}^{(2)}
 \right)^\top,
 \]
 then it holds that 
 \begin{align}\label{eq: v_expansion}
 \boldsymbol{v}(x)=\sum_{\bm{l}\in\mathbb{Z}^2} \boldsymbol{v}_{\bm{l}}\phi_{\bm{l}}(x),
 \end{align}
 with the Fourier coefficients rewritten by
 \begin{align}\label{eq: Fourier_coe}
 \bm{v}_{\bm{l}}=\frac{1}{a^2}\int_\Omega\boldsymbol{v}(x)\overline{\phi_{\bm{l}}(x)}\mathrm{d}x,
 \quad \bm{l}\in\mathbb{Z}^2.
 \end{align}
 
 Then, the mean function $\boldsymbol{g}$  and the variance function $\boldsymbol{\sigma}^2$ can be represented by the Fourier series 
 \begin{align}\label{eq: elastic_g}
\boldsymbol{g}(x)=\sum_{\bm{l}\in\mathbb{Z}^2}\widehat{\boldsymbol{g}}_{\boldsymbol l}\phi_{\boldsymbol l}(x), \\\label{eq: elastic_sigma}
\boldsymbol{\sigma}^2(x)=\sum_{\bm{l}\in\mathbb{Z}^2}\widehat{\boldsymbol{\sigma}}_{\boldsymbol l}\phi_{\boldsymbol l}(x),
 \end{align}
 where $\widehat{\boldsymbol{g}}_{\boldsymbol l}$ and $\widehat{\boldsymbol{\sigma}}_{\boldsymbol l}$ are the Fourier coefficients. 
 
 In the next two subsections, we shall investigate the numerical method to reconstruct the Fourier coefficients $\widehat{\boldsymbol{g}}_{\boldsymbol l}$ and $\widehat{\boldsymbol{\sigma}}_{\boldsymbol l}$  from the measurements. 
 
 \subsection{Recover the mean function}
This subsection is devoted to recovering the mean function $\boldsymbol{g}$ from $\mathbf{E}$. Analogous to the acoustic case, we first define the admissible angular frequencies.

 \begin{definition}[Admissible angular frequencies to reconstruct the mean function $\boldsymbol{g}$]
Let $\xi$ be a sufficiently small positive constant, then the admissible angular frequencies can be defined by
\begin{align}\label{eq: omega}
\omega_{\bm l}=\left\{
\begin{aligned}
&\frac{2\pi}{a}|\bm l|,&&{\bm l}\in\mathbb{Z}^2\backslash\{\boldsymbol 0\},\\
&\frac{2\pi}{a}\xi,&&{\bm l}=\{\bm 0\}.
\end{aligned}
\right.
\end{align}
The corresponding observation direction is defined by 
\begin{align}\label{eq: xhat_elastic}
\widehat{x}_{\bm l}=\left\{
\begin{aligned}
&\frac{\bm l}{|\bm l|},&&{\bm l}\in\mathbb{Z}^2\backslash\{\boldsymbol{0}\},\\
&(1,0)^\top,&&{\bm l}={\bm 0}.
\end{aligned}
\right.
\end{align}
\end{definition}
 
Under such a definition, we derive the following results:
 \begin{theorem}
   Let the angular frequencies and the observation directions be defined by \eqref{eq: omega} and \eqref{eq: xhat_elastic}, respectively. Then for ${\bm l}\in\mathbb{Z}^2$, the Fourier coefficients $\widehat{\bm g}_{\bm l}$ of the deterministic source term $\bm g$ in \eqref{eq: elastic_g} can be determined by the multi-frequency mean data $\mathbf{E},$ i.e., $\left\{
     \mathbf{E}\left[\bm{u}_\xi^\infty(\widehat{x}_{\bm l};c_\xi\omega_{\bm l})\right]: \xi\in\{p,s\}
     \right\}.$ 
 \end{theorem}
 \begin{proof}
 To indicate the dependence on $\bm l$, denote by $\gamma_{\xi,\bm l}=\frac{\mathrm{e}^{\mathrm{i}\pi/4}}{\sqrt{8\pi k_{\xi,\bm l}}},\,\xi\in\{p,s\}.$
 From the far-field patterns \eqref{eq: up_far_E} and \eqref{eq: us_far_E}, we deduce that for $\bm{l}\in\mathbb{Z}^2\backslash\{\bm 0\},$
\begin{align*}
   &\quad\frac{c_p^2}{a^2\gamma_{p,{\bm l}}}\mathbf{E}\left[\bm{u}_p^\infty(\widehat{x}_{\bm l};c_p\omega_{\bm l})\right]+
   \frac{c_s^2}{a^2\gamma_{s,{\bm l}}}\mathbf{E}\left[{\bm u}_s^\infty(\widehat{x}_{\bm l};c_s\omega_{\bm l})\right]\\
   &=\frac{1}{a^2}\int_\Omega\widehat{x}\widehat{x}^\top\mathrm{e}^{-\mathrm{i}\omega_{\bm l}\widehat{x}_{\bm l}\cdot y}{\bm g}(y)\mathrm{d}y+\frac{1}{a^2}\int_\Omega(\mathbb{I}-\widehat{x}\widehat{x}^\top)\mathrm{e}^{-\mathrm{i}\omega_{\bm l}\widehat{x}_{\bm l}\cdot y}{\bm g}(y)\mathrm{d}y\\
 &=\frac{1}{a^2}\int_\Omega{\bm g}(y)\mathrm{e}^{-\mathrm{i}\omega_{\bm l}\widehat{x}_{\bm l}\cdot y}\mathrm{d}y\\
 &=\frac{1}{a^2}\int_\Omega{\bm g}(y)\mathrm{e}^{-\mathrm{i}\left(\frac{2\pi}{a}|{\bm l}|\right)\frac{\bm l}{|\bm l|}\cdot y}\mathrm{d}y\\
  &=\frac{1}{a^2}\int_\Omega{\bm g}(y)\mathrm{e}^{-\mathrm{i}\frac{2\pi}{a}{\bm l}\cdot y}\mathrm{d}y\\
  &=\widehat{\boldsymbol g}_{\bm l},
\end{align*} 
 which leads to the computational formula
\begin{align}\label{eq: coe_gl}
 \widehat{\bm g}_{\bm l}=\frac{1}{a^2}\left(\frac{c_p^2}{\gamma_{p, {\bm l}}}\mathbf{E}\left[\bm{u}_p^\infty(\widehat{x}_{\bm l};c_p\omega_{\bm l})\right]+
 \frac{c_s^2}{\gamma_{s,{\bm l}}}\mathbf{E}\left[{\bm u}_s^\infty(\widehat{x}_{\bm l};c_s\omega_{\bm l})\right]\right).
\end{align}
 
We now consider the case ${\bm l}={\bm0}.$ Similarly to the analysis above, we can obtain that 
 \begin{align*}
 &\quad\frac{c_p^2}{a^2\gamma_{p,{\bm 0}}}\mathbf{E}\left[\bm{u}_p^\infty(\widehat{x}_{\bm 0};c_p\omega_{\bm 0})\right]+
 \frac{c_s^2}{a^2\gamma_{s,{\bm 0}}}\mathbf{E}\left[{\bm u}_s^\infty(\widehat{x}_{\bm 0};c_s\omega_{\bm 0})\right]\\
  &=\frac{1}{a^2}\int_\Omega{\bm g}(y)\mathrm{e}^{-\mathrm{i}\omega_{\bm 0}\widehat{x}_{\bm 0}\cdot y}\mathrm{d}y\\
  &=\frac{1}{a^2}\int_\Omega\left(\widehat{\boldsymbol{g}}_{\boldsymbol 0}+\sum_{\bm{l}\in\mathbb{Z}^2\backslash\{\bm0\}}\widehat{\boldsymbol{g}}_{\boldsymbol l}\phi_{\boldsymbol l}(x)\right)\overline{\phi_{\bm0}(y)}\mathrm{d}y\\
  &=\frac{1}{a^2}\widehat{\bm{g}}_{\bm 0}\int_\Omega\overline{\phi_{\bm0}(y)}\mathrm{d}y
  +\frac{1}{a^2}\sum_{\bm{l}\in\mathbb{Z}^2\backslash\{\bm0\}}\widehat{\boldsymbol{g}}_{\boldsymbol l}\int_\Omega\phi_{\boldsymbol l}(x)\overline{\phi_{\bm0}(y)}\mathrm{d}y\\
  &=\frac{\sin\xi\pi}{\xi\pi}\widehat{\boldsymbol{g}}_{\boldsymbol 0}+\frac{1}{a^2}\sum_{\bm{l}\in\mathbb{Z}^2\backslash\{\bm0\}}\widehat{\boldsymbol{g}}_{\boldsymbol l}\int_\Omega\phi_{\bm l}(x)\overline{\phi_{\bm0}(y)}\mathrm{d}y,
 \end{align*}
 which gives rise to the computational formula for $\widehat{\bm g}_{\bm 0},$ i.e.,
 \begin{align*}
   \widehat{\bm g}_{\bm 0} & =\frac{\xi\pi}{a^2\sin\xi\pi}\Bigg(
   \frac{c_p^2}{a^2\gamma_{p,{\bm 0}}}\mathbf{E}\left[\bm{u}_p^\infty(\widehat{x}_{\bm 0};c_p\omega_{\bm 0})\right]+
 \frac{c_s^2}{a^2\gamma_{s,{\bm 0}}}\mathbf{E}\left[{\bm u}_s^\infty(\widehat{x}_{\bm 0};c_s\omega_{\bm 0})\right]\\
 &\quad -\sum_{\bm{l}\in\mathbb{Z}^2\backslash\{\bm0\}}\widehat{\boldsymbol{g}}_{\boldsymbol l}\int_\Omega\phi_{\boldsymbol l}(x)\overline{\phi_{\bm0}(y)}\mathrm{d}y
   \Bigg),
 \end{align*}
 which completes the proof of this theorem.
 \end{proof}
 
 The above theorem gives rise to the approximate for the mean function by truncated Fourier series
 \begin{align*}
 \boldsymbol{g}_N(x)=\sum_{|\bm{l}|_\infty\le N}\widehat{\boldsymbol{g}}_{\boldsymbol l}\phi_{\boldsymbol l}(x).
 \end{align*}
 \subsection{Recover the variance function}
 
 In this subsection, we shall devote to investigating the numerical method to recover the variance function. 
 The admissible angular frequency to recover the variance function is introduced as follows:
 \begin{definition}
   [Admissible angular frequency to recover the variance function]
   Take $\omega_0>0$ to be a sufficiently small positive constant, the admissible frequencies $\{\tau_{\bm l}\}$ for each ${\bm l}\in\mathbb{Z}^2$ is defined by
   \begin{align}\label{eq: tau_elastic_variance}
   \tau_{\bm l}=\frac{2\pi}{a}|\bm l|,\quad {\bm l}\in\mathbb{Z}^2.
   \end{align}
   The corresponding observation is defined by 
   \begin{align}\label{eq: xhat_elastic_variance}
\widehat{x}_{\bm l}:=
\left\{
\begin{aligned}
&\frac{\bm l}{|\bm l|},&&{\bm l}\ne{\bm0},\\
&{\bm a},&&{\bm l}={\bm0},
\end{aligned}
\right.
\end{align}
with ${\bm  a}=(a_1,a_2)\in\mathbb{S}$ being some unit vector chosen flexibly.
 \end{definition}
 By choosing admissible angular frequencies $\{\tau\}$ and some $\omega_0>0,$ we can derive the following results:
 \begin{theorem}
   Let $\omega_0>0$ be a small positive angular frequency, and select the admissible angular frequencies $\{\tau_{\bm l}\}$ as well as the observation directions $\widehat{x}_{\bm l}$ as \eqref{eq: tau_elastic_variance} and \eqref{eq: xhat_elastic_variance}, respectively. Then for each ${\bm l}\in\mathbb{Z}^2,$ the Fourier coefficients $\widehat{\bm \sigma}_{\bm l}$ of the deterministic sigma function $\bm\sigma$ in \eqref{eq: elastic_sigma} can be determined by the multi-frequency covariance data $\mathbf{C},$ i.e.,$\mathbf{C}\left[\bm{u}_\alpha^\infty(\widehat{x}_{\bm l};c_\alpha(\omega_0+\tau_{\bm l})),\bm{u}_\beta^\infty(\widehat{x}_{\bm l};c_\beta(\omega_0+\tau_{\bm l}))\right]$ with
$(\alpha,\beta)\in\{(p,p),(p,s),(s,p),(s,s)\}$
 \end{theorem}
 \begin{proof}
For a fixed angular frequency $\omega>0,$  denote
\begin{align}\label{eq: UpUs}
\bm{u}_\xi(\widehat{x},\omega)=\frac{c_\xi^2}{\gamma_{\xi}}\bm{u}_\xi^\infty(\widehat{x},c_\xi\omega),
\quad\xi\in\{p,s\},
\end{align}
then it holds that 
\begin{align}\label{eq: Up}
  {\bm U}_p(\widehat{x},\omega_0+\tau_{\bm l})&=\widehat{x}\widehat{x}^\top\int_\Omega\mathrm{e}^{-\mathrm{i}(\omega_0+\tau_{\bm l})\widehat{x}\cdot y}\left(
  \boldsymbol{g}(y)+\boldsymbol{\sigma}(y)\dot{W}_y
   \right)\mathrm{d}y,\\\label{eq: Us}
  {\bm U}_s(\widehat{x},\omega_0+\tau_{\bm l})&=\left(\mathbb{I}-\widehat{x}\widehat{x}^\top\right)\int_\Omega\mathrm{e}^{-\mathrm{i}(\omega_0+\tau_{\bm l})\widehat{x}\cdot y}\left(
  \boldsymbol{g}(y)+\boldsymbol{\sigma}(y)\dot{W}_y
   \right)\mathrm{d}y.
\end{align}
This together with \eqref{eq: up_Eup}--\eqref{eq: us_Eus} leads to
\begin{align*}
  {\bm U}_p(\widehat{x}_{\bm l};\omega_0+\tau_{\bm l})-\mathbf{E}\left[{\bm U}_p(\widehat{x}_{\bm l};\omega_0+\tau_{\bm l})\right] & =\widehat{x}_{\bm l}\widehat{x}_{\bm l}^\top\int_\Omega\mathrm{e}^{-\mathrm{i}(\omega_0+\tau_{\bm l})\widehat{x}_{\bm l}\cdot y}\boldsymbol{\sigma}(y)\mathrm{d}W_y,\\
  {\bm U}_s(\widehat{x}_{\bm l};\omega_0+\tau_{\bm l})-\mathbf{E}\left[{\bm U}_s(\widehat{x}_{\bm l};\omega_0+\tau_{\bm l})\right] & =\left(\mathbb{I}-\widehat{x}_{\bm l}\widehat{x}_{\bm l}^\top\right)\int_\Omega\mathrm{e}^{-\mathrm{i}(\omega_0+\tau_{\bm l})\widehat{x}_{\bm l}\cdot y}\boldsymbol{\sigma}(y)\mathrm{d}W_y.
\end{align*}
Denote ${\bm U}={\bm U}_p+{\bm U}_s$, then it holds that
\[
{\bm U}(\widehat{x}_{\bm l},\omega_0+\tau_{\bm l})-\mathbf{E}\left[{\bm U}(\widehat{x}_{\bm l},\omega_0+\tau_{\bm l})\right]
=\int_\Omega\mathrm{e}^{-\mathrm{i}(\omega_0+\tau_{\bm l})\widehat{x}_{\bm l}\cdot y}\boldsymbol{\sigma}(y)\mathrm{d}W_y.
\]
Then we can obtain that 
\begin{align}\nonumber
  &\quad\frac{1}{a^2}\mathbf{C}\left[{\bm U}(\widehat{x}_{\bm l};\omega_0+\tau_{\bm l}),{\bm U}(\widehat{x}_{\bm l};\omega_0)\right]\\\nonumber
  &=\frac{1}{a^2}\mathbf{E}\left[\left({\bm U}\left(\widehat{x}_{\bm l};\omega_0+\tau_{\bm l}\right)-\mathbf{E}\left[{\bm U}\left(\widehat{x}_{\bm l};\omega_0+\tau_{\bm l}\right)\right]\right)\overline{\left({\bm U}(\widehat{x}_{\bm l};\omega_0)-
  \mathbf{E}\left[{\bm U}(\widehat{x}_{\bm l};\omega_0)\right]\right)}\right]\\\nonumber
  &=\frac{1}{a^2}\int_\Omega\int_\Omega\mathbf{E}\left[{\bm\sigma}(y)\dot{W}_y\overline{{\bm\sigma}(z)\dot{W}_z}\right]\mathrm{e}^{-\mathrm{i}(\omega_0+\tau_{\bm l})\widehat{x}_{\bm l}\cdot y}\mathrm{e}^{\mathrm{i}\omega_0\widehat{x}_{\bm l}\cdot y}\mathrm{d}y\mathrm{d}z\\
  &=\frac{1}{a^2}\int_\Omega{\bm\sigma}^2(y)\mathrm{e}^{-\mathrm{i}\tau_{\bm l}\widehat{x}_{\bm l}\cdot y}\mathrm{d}y=\widehat{\bm \sigma}_{\bm l},\label{eq: coe_sigmal}
\end{align}
where we have taken \eqref{eq: property} into consideration. Specifically, for ${\bm l}={\bm 0},$ 
$$
\widehat{\bm \sigma}_{\bm 0}=\frac{1}{a^2}\mathbf{C}\left[{\bm U}({\bm a};\omega_0),{\bm U}({\bm a};\omega_0)\right].
$$
 \end{proof}

Now, we obtain the approximate for the variance function by 
$$
 \bm{\sigma}_N^2(x)=\sum_{|\bm{l}|_\infty\le N}\widehat{\bm{\sigma}}_{\bm l}\phi_{\bm l}(x).
$$


 \subsection{Numerical simulations}
 In this subsection, several numerical examples will be presented to illustrate the performance of the proposed Fourier method in reconstructing the statistics for the elastic source. Similarly to the numerical implementation for the acoustic wave, the integral domain $\Omega$ is set to be $[-0.5,0.5]\times[-0.5,0.5].$ The total number of realizations is $10^6.$  To test the stability, the noisy far field data is given by 
 $$
 {\bm u}_{\xi,\delta}^\infty={\bm u}_\xi^\infty+\delta r_3|{\bm u}_\xi^\infty|\mathrm{e}^{\mathrm{i}\pi r_4},
 $$
 with $\xi\in\{p,s\},$ and $r_\ell,\,\ell=3,4$ being two uniformly distributed random numbers ranging from $-1$ to 1, and $\delta>0$ is the noise level. 
 In the following, the Lam\'{e} constants are chosen as $\lambda=\mu=1.$ The angular frequencies are chosen to be 
 \begin{align*}
\omega_{\bm l}=\left\{
\begin{aligned}
&{2\pi}|\bm l|,&&1\le|{\bm l}|_\infty\le 20\\
&{2\pi}\times10^{-3},&&{\bm l}={\bm 0}
\end{aligned}
\right.
,\ \
\tau_{\bm l}=\frac{2\pi}{a}|\bm l|,\quad 1\le|{\bm l}|_\infty\le 20.
 \end{align*}
 
Under these admissible angular frequencies, the artificial far-field data can be given by
\begin{align*}
    \left\{
     \mathbf{E}\left[\bm{u}_{\xi,\delta}^\infty(\widehat{x};c_\xi\omega_{\bm l})\right]: \xi\in\{p,s\}
     \right\},
\end{align*}
and 
\begin{align*} 
\left\{  \mathbf{C}\left[\bm{u}_{\alpha,\delta}^\infty(\widehat{x};c_\alpha(\omega_0+\tau_{\bm l})),\bm{u}_{\beta,\delta}^\infty(\widehat{x};c_\beta(\omega_0+\tau_{\bm l}))\right]
:(\alpha,\beta)\in\{(p,p),(p,s),(s,p),(s,s)\}
\right\},
\end{align*}
 where $\omega_0>0$ is some small angular frequency chosen flexibly. In this section, we take $\omega_0=0.001.$
 
 Once the measurements are available, the Fourier coefficients $\widehat{\bm g}_{\bm l}^\delta$ and $\widehat{\bm \sigma}_{\bm l}^\delta $ can be obtained through \eqref{eq: coe_gl} and \eqref{eq: coe_sigmal} provided that the measurements are changed into their noisy counterpart.
 
 We now consider the reconstruction of two random sources with the following pairs of mean and variance 
\begin{align*}
& \begin{cases}
g_1(x_1, x_2)=\exp(-200((x_1-0.01)^2+(x_2-0.12)^2))\\
\qquad\qquad\quad\ -100(x_2^2-x_1^2)\exp(-90(x_1^2+x_2^2)),\\
\sigma_1(x_1,x_2)=\frac{1}{2}g_1(x_1,x_2), 
\end{cases} \\
&  \begin{cases}
g_2(x_1, x_2)=1500x_1^2 x_2\exp(-50(x_1^2+x_2^2)),\\
\sigma_2(x_1,x_2)=\frac{1}{2}\exp(-200((x_1-0.01)^2+(x_2-0.12)^2)).
\end{cases} 
\end{align*}
 
The exact source functions are plotted in \Cref{fig: source}.
\begin{figure}
    \centering
    \subfigure[$g_1$]{\includegraphics[width=0.4\linewidth]{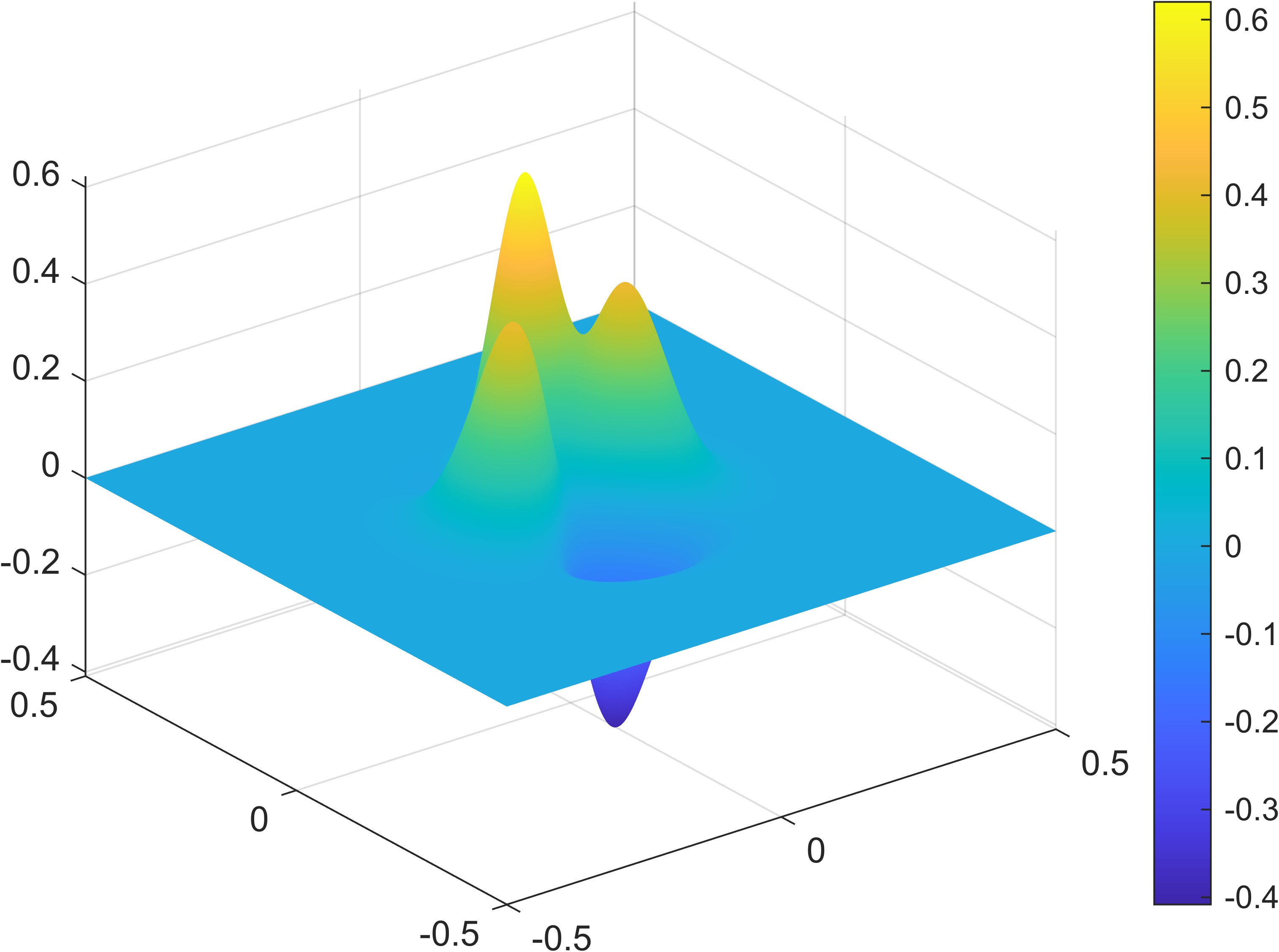}}\qquad
    \subfigure[$g_2$]{\includegraphics[width=0.4\linewidth]{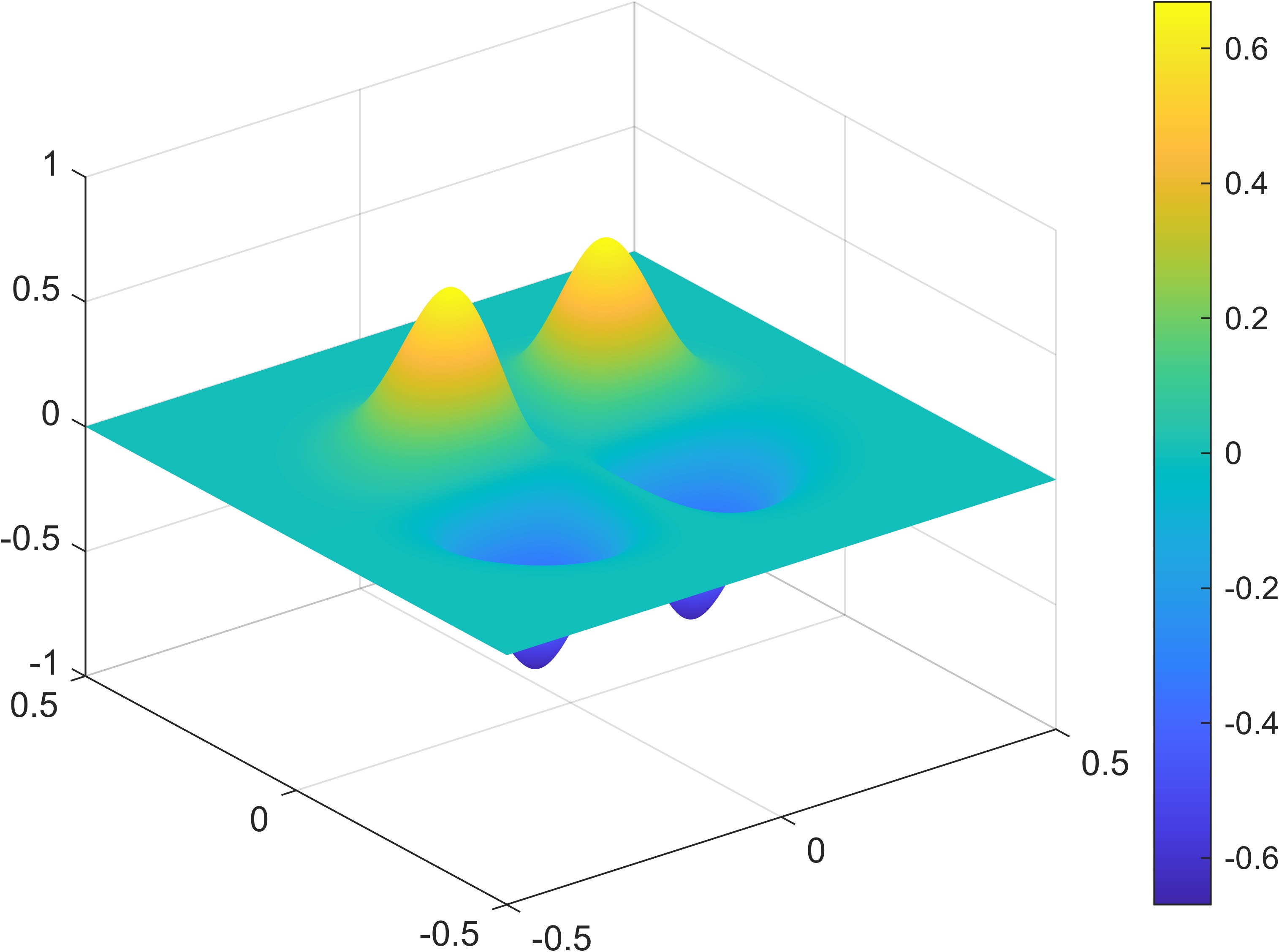}}
    \subfigure[$\sigma_1^2$]{\includegraphics[width=0.4\linewidth]{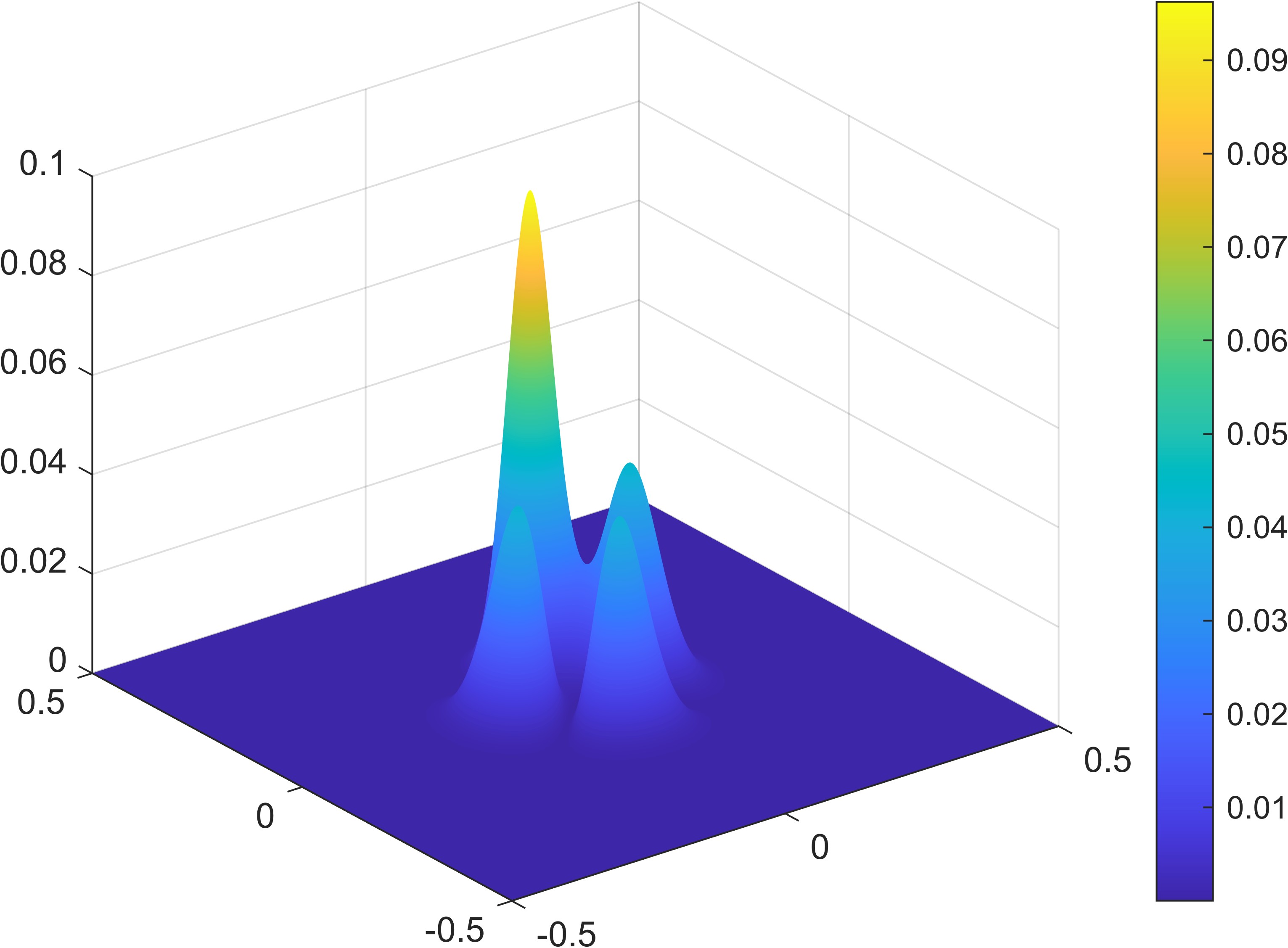}}\qquad
    \subfigure[$\sigma_2^2$]{\includegraphics[width=0.4\linewidth]{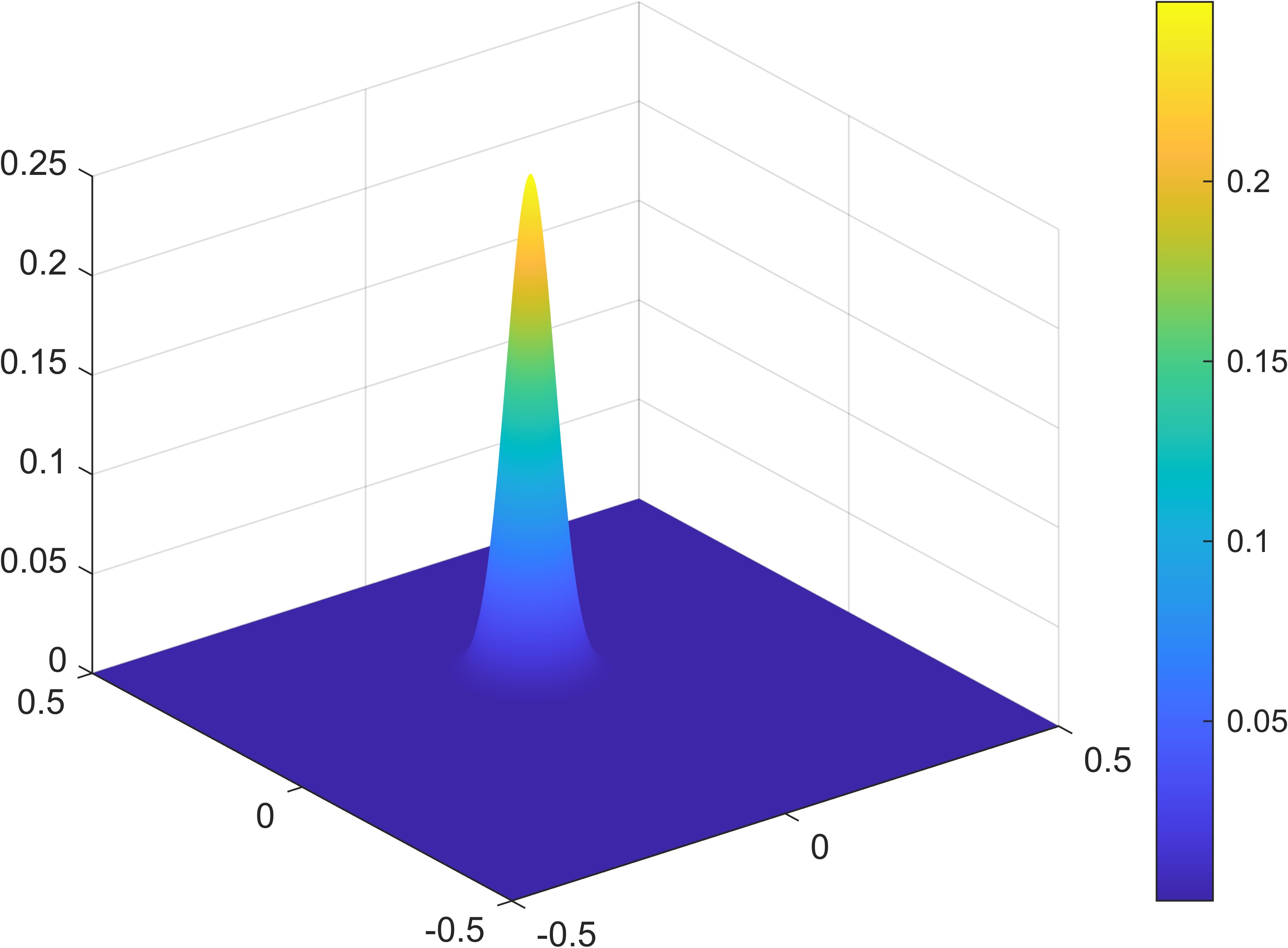}}
     \caption{The exact mean functions and variance functions.}\label{fig: source}
\end{figure}
 
The reconstruction with $\delta = 5\%$ is displayed in \Cref{fig: elastic1}. \Cref{fig: source} and \Cref{fig: elastic1} illustrate that all the means and variances are well recovered. By taking $\delta = 20\%,$ we plot the reconstruction in \Cref{fig: elastic2}. By comparing \Cref{fig: elastic1} with \Cref{fig: elastic2}, we find that the reconstruction quality increases as the noise level decreases.
 
 Nevertheless, our method performs well in reconstructing the elastic source.
 By taking different noise levels, we list the reconstruction errors in \Cref{tab: elastic}. We can observe from \Cref{tab: elastic} that our method can realize an accurate reconstruction, which further illustrates the powerful ability to recover the elastic source function. Though the $H_1$ reconstruction error of the variance function ${\bf \sigma}^2$ is up to $6.32\%$ when the noise level is $10\%,$ we can see from \Cref{fig: elastic2} that even $\delta=20\%,$ the reconstruction is acceptable, except that there may be some fluctuations in function values in flat regions.
 
\begin{table}
  \centering
  \begin{tabular}{ccccc}
    \toprule
    $\delta$ & $0.5\%$ & $1\%$ & $5\%$ &$10\%$ \\
    \midrule
    \vspace{0.2cm}
    ${\|{\bm g}^\delta_N-{\bm g}\|_0}/{\|{\bm g}\|_0}$&$1.27\%$&$1.27\%$&$1.80\%$&$3.13\%$\\
    \vspace{0.2cm}
    ${\|{\bm g}^\delta_N-{\bm g}\|_1}/{\|{\bm g}\|_1}$&$3.48\%$&$3.51\%$&$4.27\%$&$5.77\%$\\
    \vspace{0.2cm}
    ${\|{\bm \sigma}^{2,\delta}_N-{\bm\sigma}^2\|_0}/{\|{\bm\sigma}^2\|_0}$&$0.58\%$&$0.59\%$&$1.75\%$&$3.28\%$\\
    ${\|{\bm \sigma}^{2,\delta}_N-{\bm\sigma}^2\|_1}/{\|{\bm\sigma}^2\|_1}$&$1.72\%$&$1.79\%$&$3.34\%$&$6.32\%$\\
    \bottomrule
  \end{tabular}
  \caption{The relative errors of the reconstruction of ${\bm g}_N$ and ${\bm\sigma}_N^2$ for different noise levels.}\label{tab: elastic}
\end{table}

\begin{figure}
  \centering
   \subfigure[$g_{N,1}$]{\includegraphics[width=0.45\linewidth]{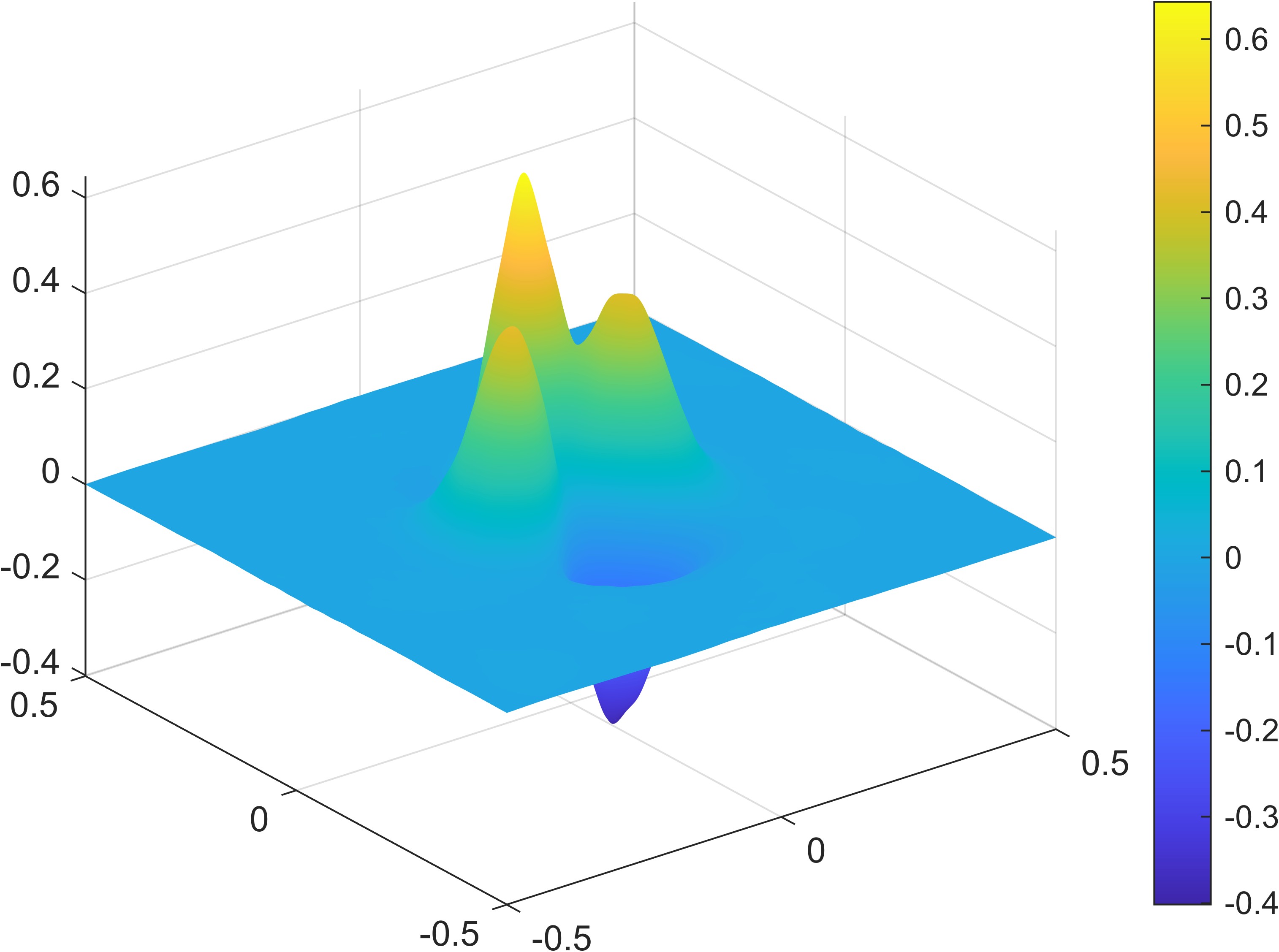}}\qquad 
   \subfigure[$g_{N,2}$]{\includegraphics[width=0.45\linewidth]{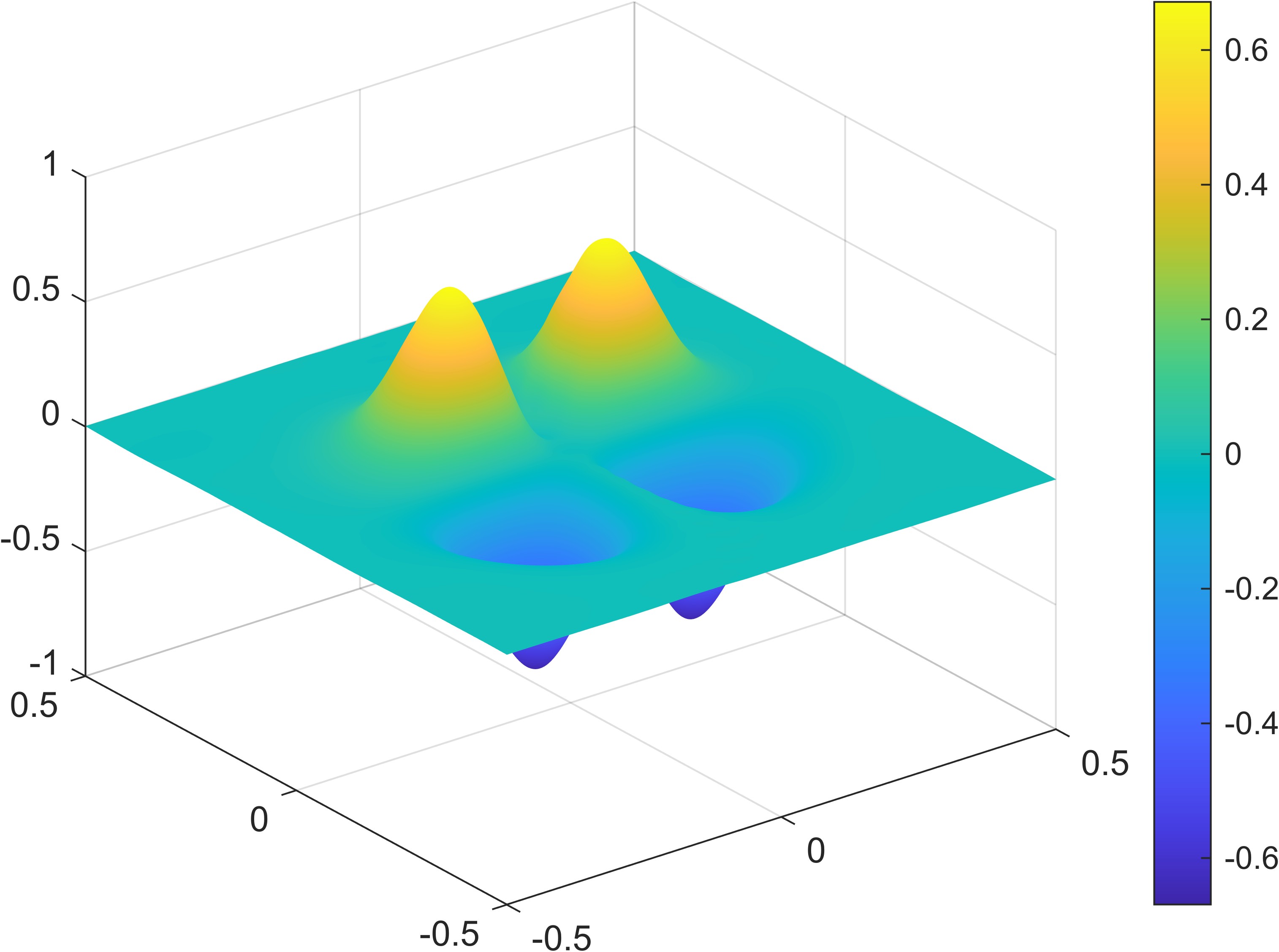}}
   \subfigure[$\sigma_{N,1}^2$]{\includegraphics[width=0.45\linewidth]{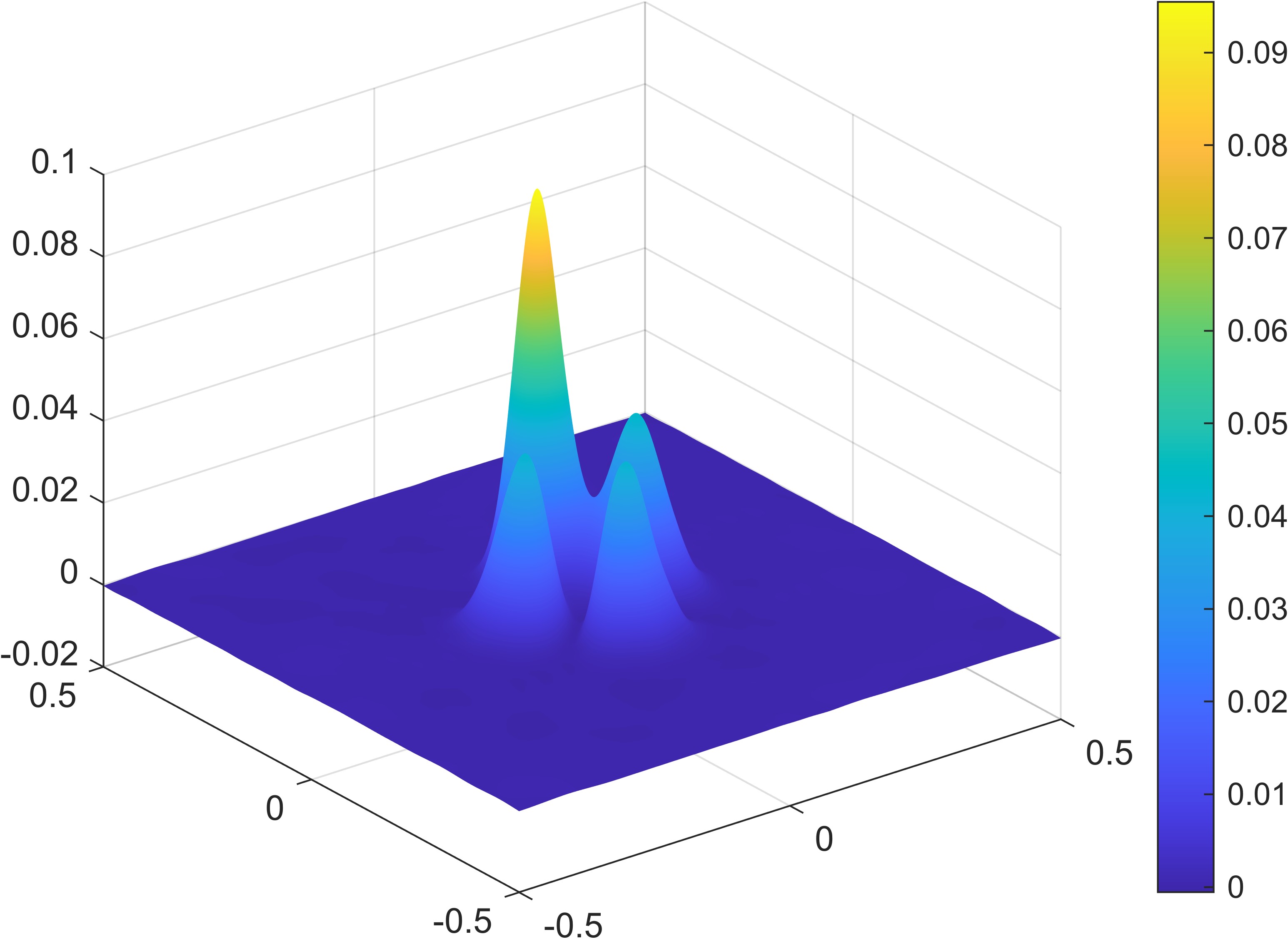}}\qquad 
   \subfigure[$\sigma_{N,2}^2$]{\includegraphics[width=0.45\linewidth]{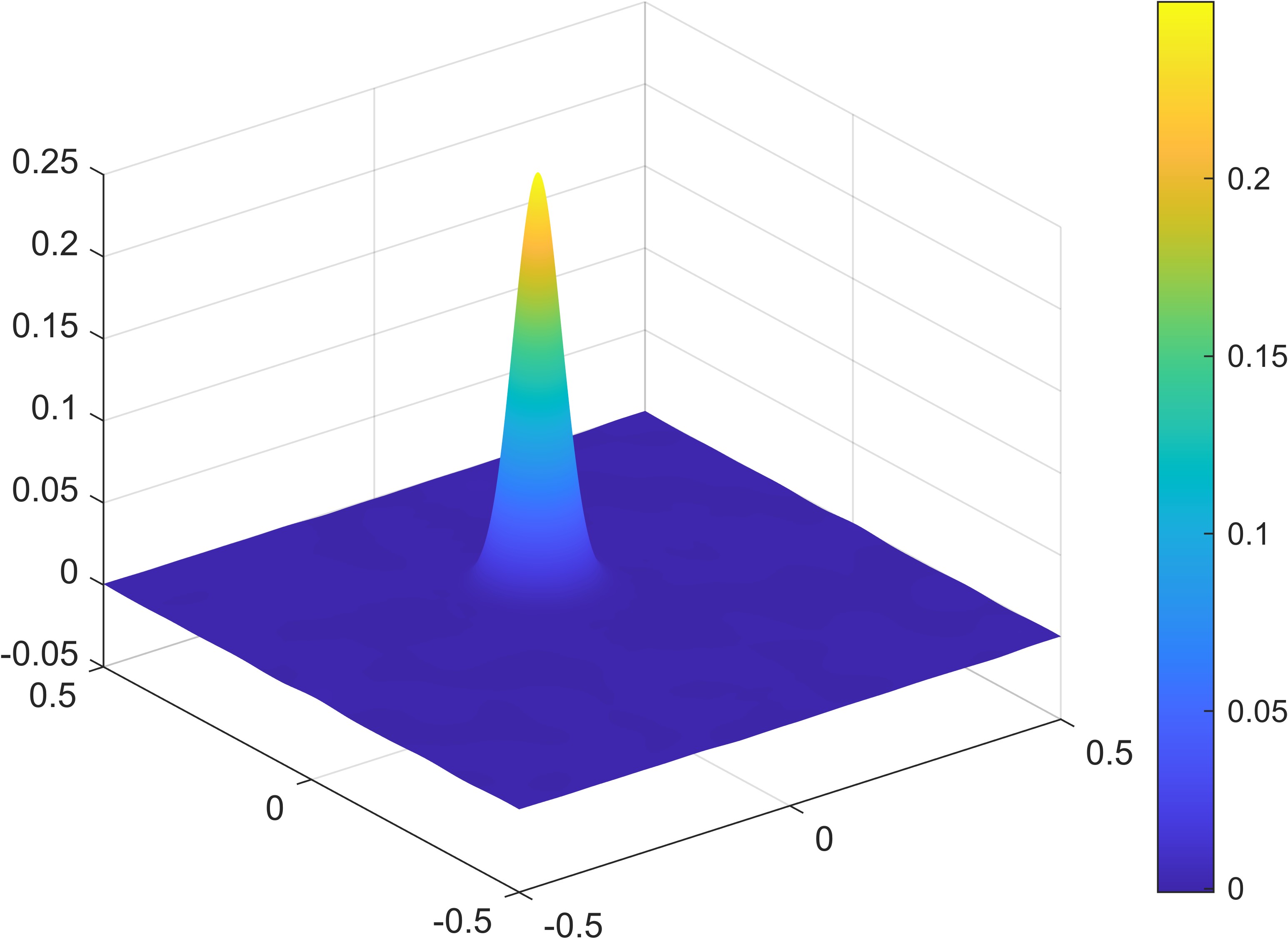}}
  \caption{The reconstructed mean functions and variance functions for $\delta = 5\%$.}\label{fig: elastic1}
\end{figure}

\begin{figure}
    \centering
    \subfigure[$g_{N,1}$]{\includegraphics[width=0.45\linewidth]{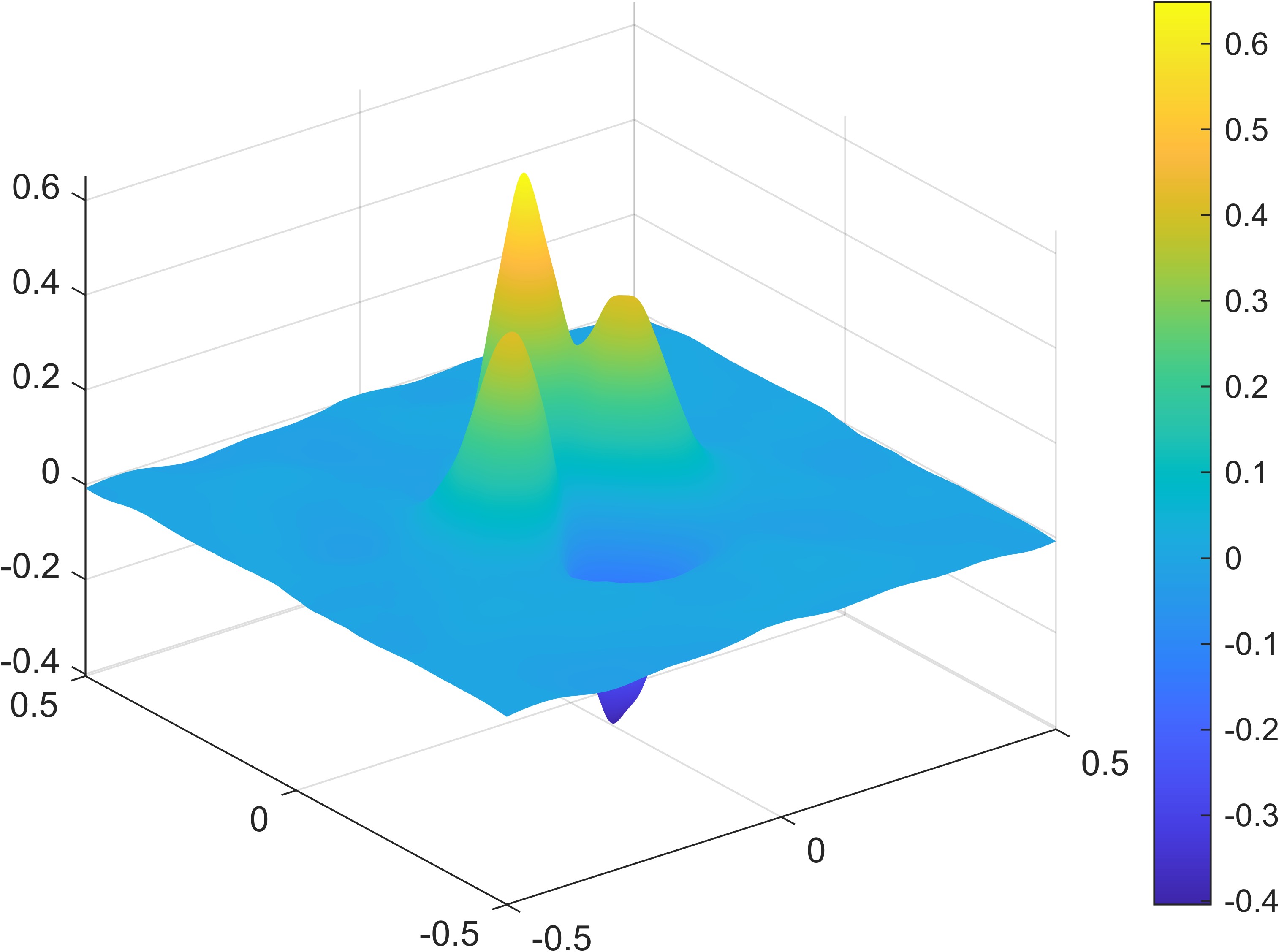}}\qquad 
    \subfigure[$g_{N,2}$]{\includegraphics[width=0.45\linewidth]{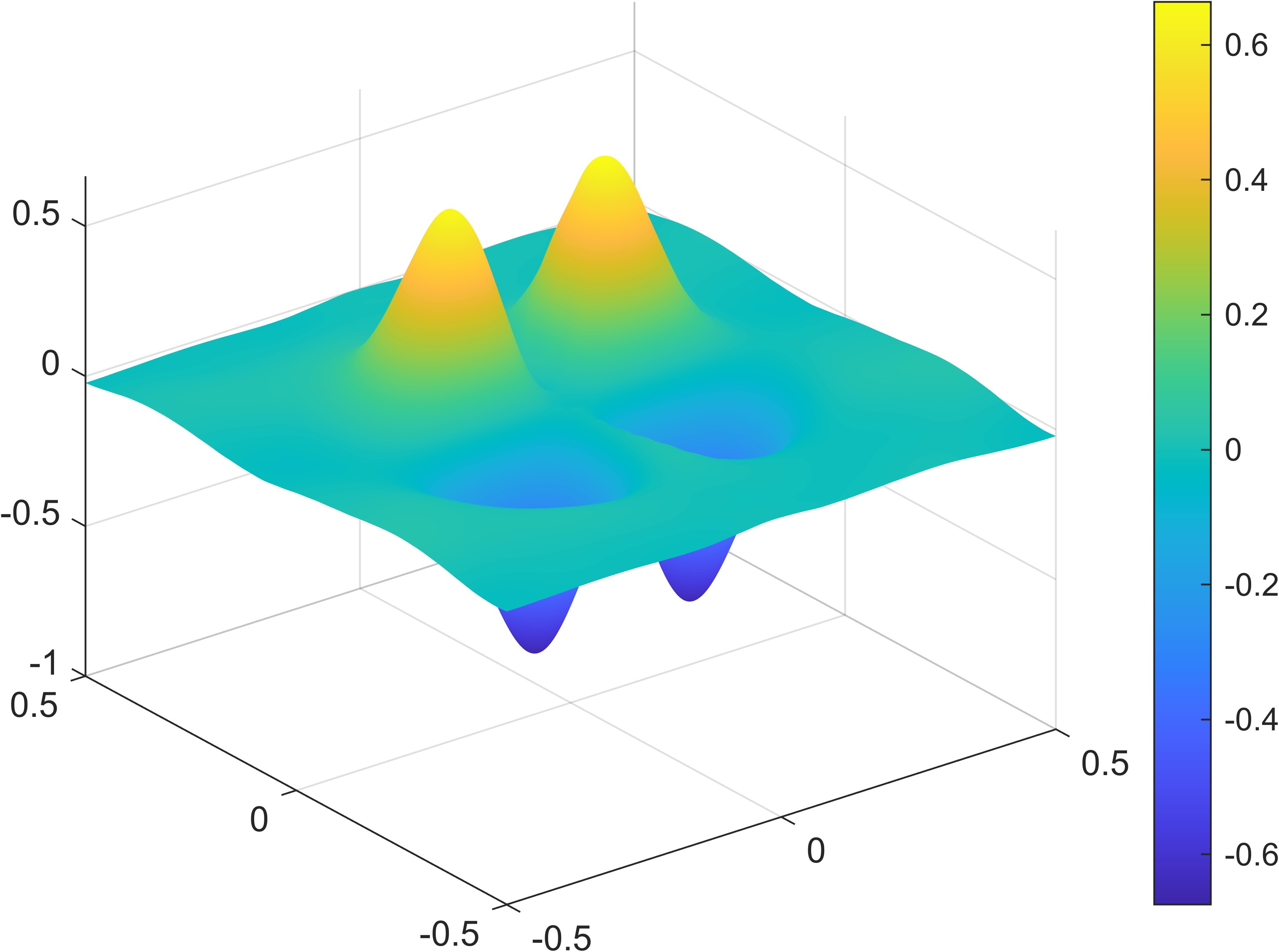}}
    \subfigure[$\sigma_{N,1}^2$]{\includegraphics[width=0.45\linewidth]{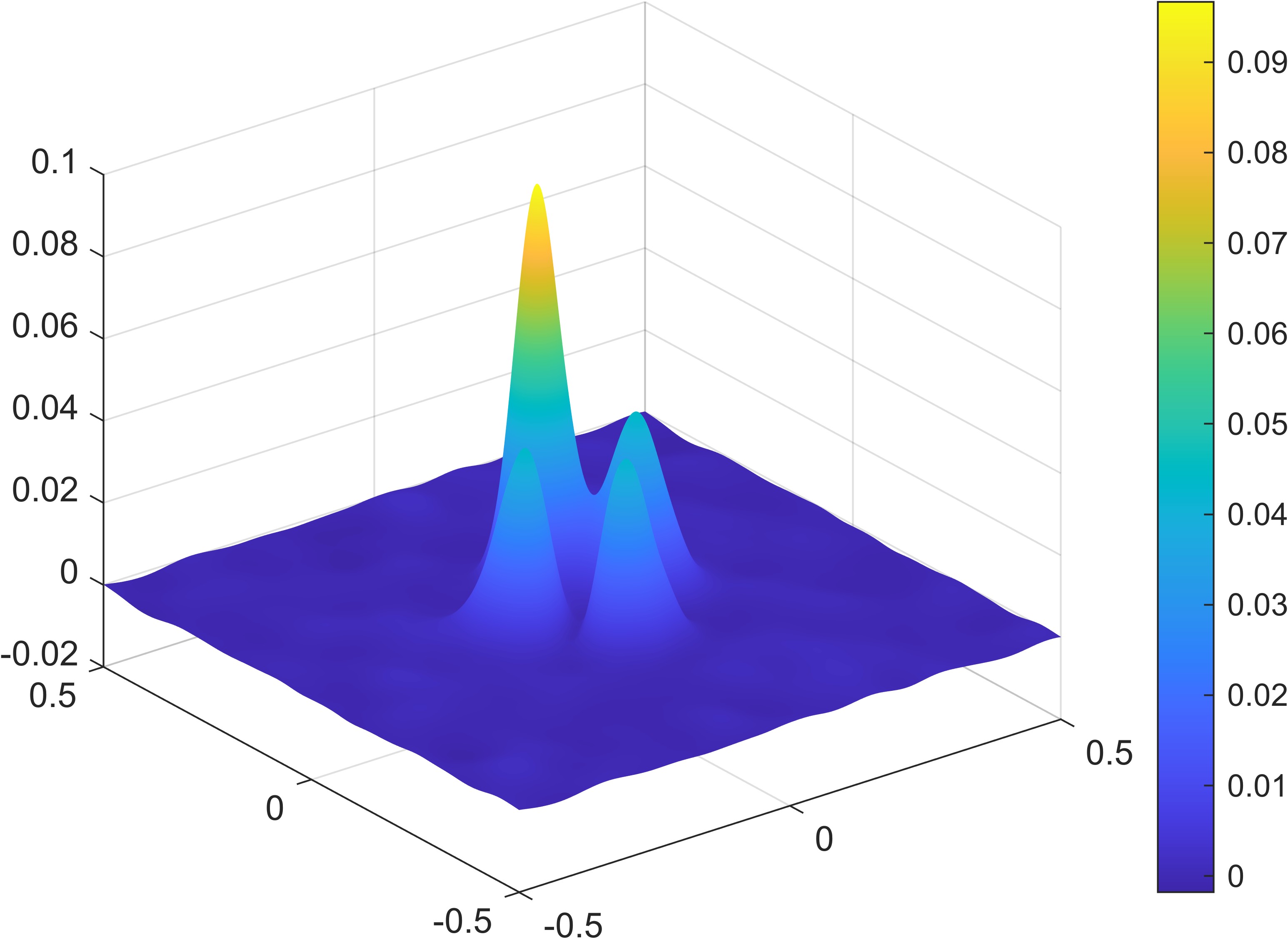}}\qquad 
    \subfigure[$\sigma_{N,2}^2$]{\includegraphics[width=0.45\linewidth]{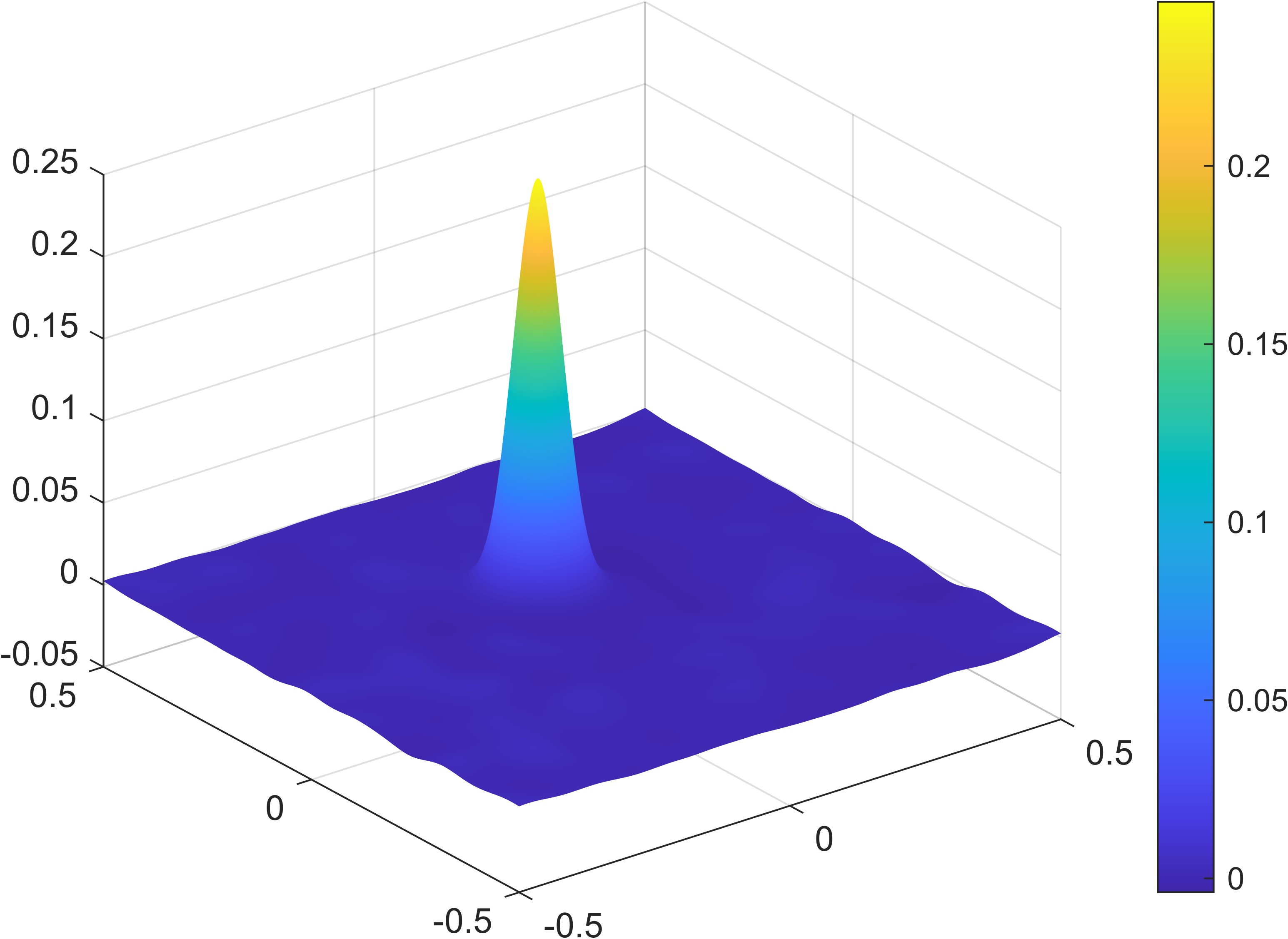}}
    \caption{The reconstructed mean functions and variance functions for $\delta = 20\%$.}\label{fig: elastic2}
\end{figure}

\section{Conclusion}\label{sec: conclusion}
 
We present a study on the inverse random source scattering problems for the acoustic and the elastic wave, where the source is driven by an additive white noise. By expanding the mean function and the variance function via the Fourier series respectively, we establish the relation between the statistical quantities and the Fourier coefficients, resulting in an easy-to-implement and effective method. A promising feature of the current work is that no iterative solver is involved and the reconstruction is highly accurate. Numerical examples in both the acoustic and elastic cases are presented to illustrate the performance of the proposed method.





\bibliographystyle{plain}
\bibliography{references}
\end{document}